\documentclass[twoside,10pt,a4paper]{article}
\usepackage[colorlinks=true]{hyperref}                        
\hypersetup{ linkcolor=blue,  filecolor=black,  urlcolor=red,  citecolor=black} 

\usepackage{amsmath,latexsym,amsfonts,indentfirst,amsthm,amsxtra,amssymb,bm}
\usepackage[numbers,sort&compress]{natbib}
\usepackage[perpage,symbol*]{footmisc}
\usepackage{relsize}

\numberwithin{equation}{section}

 \newtheorem{lemma}{Lemma}[section]
 \newtheorem{proposition}[lemma]{Proposition}
 \newtheorem{theorem}{Theorem}[section]
 \newtheorem{corollary}[lemma]{Corollary}

 \theoremstyle{remark}
 \newtheorem{remark}{Remark}[section]

\usepackage{accents}  
\makeatletter
\def\wideubar{\underaccent{{\cc@style\underline{\mskip10mu}}}}
\def\Wideubar{\underaccent{{\cc@style\underline{\mskip8mu}}}}
\makeatother
\makeatletter
\def\widebar{\accentset{{\cc@style\underline{\mskip10mu}}}}
\def\Widebar{\accentset{{\cc@style\underline{\mskip8mu}}}}
\makeatother

\newcommand{\lrn}{\left\vert\kern-0.3ex\left\vert\kern-0.3ex\left\vert}
\newcommand{\rrn}{\right\vert\kern-0.3ex\right\vert\kern-0.3ex\right\vert}

\begin{document}

\title{
 \bf Symmetric flows for compressible heat-conducting fluids with temperature dependent viscosity coefficients}
\author{
{\sc Ling Wan}\\
{\footnotesize School of Mathematics and Statistics, Wuhan University, Wuhan 430072, China}\\[2mm]
{\sc Tao Wang}\thanks{Corresponding author. Email addresses: ling.wan@whu.edu.cn (L.\,Wan), tao.wang@whu.edu.cn (T. Wang).}\\
{\footnotesize School of Mathematics and Statistics, Wuhan University, Wuhan 430072, China}\\[-1mm]
{\footnotesize DICATAM, Mathematical Division, University of Brescia, Via Valotti 9, 25133 Brescia, Italy}
}

\date{\empty}

\maketitle
\begin{abstract}
We consider the Navier--Stokes equations for compressible heat-conducting ideal polytropic gases in a bounded annular domain when the viscosity and thermal conductivity coefficients are general smooth functions of temperature.
A global-in-time, spherically or cylindrically symmetric, classical solution to the initial boundary value problem is shown to exist uniquely and converge exponentially to the constant state as the time tends to infinity under certain assumptions on the initial data and the adiabatic exponent $\gamma$.
The initial data can be large if $\gamma$ is sufficiently close to 1.
These results are of Nishida--Smoller type and extend the work [Liu et al., SIAM J. Math. Anal. 46 (2014), 2185--2228] restricted to the one-dimensional flows.

\vspace{2mm}


\vspace{2mm}
\end{abstract}


\section{ Introduction}
The motion of a compressible viscous and heat-conducting fluid in $\Omega\subset\mathbb{R}^d$ can be described by the Navier--Stokes equations in Eulerian coordinates:
\begin{subequations}
 \label{cns1}
 \begin{alignat}{2}
 &\partial_t\rho+{\rm div}(\rho \bm{u})=0,\\
 &\partial_t(\rho \bm{u})+{\rm div}(\rho \bm{u}\otimes\bm{u})
 +\nabla P={\rm div}\,\mathbb{S},\\
 &\partial_t(\rho E)+{\rm div}(\rho \bm{u}E+\bm{u}P)
 ={\rm div}(\kappa\nabla\theta+\mathbb{S}\cdot\bm{u}).
 \end{alignat}
\end{subequations}
Here  the primary dependent variables are the density $\rho$, the velocity field $\bm{u}\in\mathbb{R}^d$, and the temperature $\theta$.
The symbol $E=e+\tfrac{1}{2}|\bm{u}|^2$ is the specific total energy.
For ideal polytropic gases, the pressure $P$ and the specific internal energy $e$ are related with $\rho$ and $\theta$ by equations of state:
\begin{equation} \label{ideal1}
P=R\rho\theta,\quad e=c_v\theta,
\end{equation}
where $R>0$,  $c_v=R/(\gamma-1)$, and $\gamma>1$ are the specific gas constant, the specific heat at constant volume, and the adiabatic exponent, respectively. 
The fluid is assumed to be Newtonian so that the viscous stress tensor $\mathbb{S}$ is of form
\begin{equation} 
\label{VST}
\mathbb{S}=\mu\left(\nabla\bm{u}+(\nabla\bm{u})^{\mathsf{T}}\right)
+\lambda\,{\rm div}\,\bm{u}\,\mathbb{I},
\end{equation}
where  $(\nabla\bm{u})^{\mathsf{T}}$ is the transpose matrix of $\nabla\bm{u}$ and $\mathbb{I}$ is the $d\times d$ identity matrix.
The viscosity coefficients $\mu$, $\lambda$, and the thermal conductivity coefficient $\kappa$ are prescribed through constitutive relations as functions of the density and temperature satisfying $\mu>0$, $\kappa>0$, and $2\mu+d\lambda>0$ (see \citet{LL87MR961259}).

In this paper we establish the existence and exponential decay rate of global-in-time,  spherically or cylindrically symmetric, classical solutions to \eqref{cns1}--\eqref{VST} in the bounded annular domain $\Omega=\{\bm{x}=(x_1,\cdots\!,x_d)\in\mathbb{R}^d: a<r<b\}$ for $0<a<b<\infty$ with large initial data. Here we set $d\in\mathbb{N}_+$ and  $r=|\bm{x}|$ in the spherically symmetric case, while $d=3$ and $r=\sqrt{x^2_1+x^2_2}$ in the cylindrically symmetric case. 
The system \eqref{cns1}--\eqref{VST} is
supplemented with  the spherically or cylindrically symmetric initial data:
\begin{align} \label{initial1}
(\rho,\bm{u},\theta)(0,\bm{x})
=(\rho_0,\bm{u}_0,\theta_0)(\bm{x}) \quad {\rm for}\  \ \bm{x}\in\Omega,
\end{align}
and the boundary conditions:
\begin{equation}
\label{bdy1}
\bm{u}=0,\quad
\frac{\partial \theta}{\partial \bm{n}}=0\quad {\rm on}\  \ \partial\Omega,
\end{equation}
where $\bm{n}$ denotes the unit outward normal vector to $\partial\Omega$.

We are interested in the 
case where the transport coefficients $\mu$, $\lambda$, and  $\kappa$ are smooth functions of the temperature.
More specifically, we suppose that
\begin{equation} \label{transport}
\mu,\lambda,\kappa\in C^3(0,\infty),\quad
\mu(\theta)>0,\ \ \kappa(\theta)>0,\ \
2\mu(\theta)+d\lambda(\theta)>0 \quad {\rm for\ all}\ \ \theta>0.
\end{equation}
Our main motivation is provided by the kinetic theory of gases.
By virtue of the Chapman--Enskog expansion, the compressible Navier--Stokes system \eqref{cns1} is the first order approximation of the Boltzmann equation, and the transport coefficients $\mu$, $\lambda$, and $\kappa$ depend solely on the temperature (see \cite{CC90MR1148892} and \cite[Chapter X]{VK65}). 
In particular, if the intermolecular potential varies as $r^{-a}$ with $r$ being the molecule distance, then $\mu$, $\lambda$, and $\kappa$ satisfy
\begin{equation}\label{transport1}
 \mu=\bar{\mu}\theta^{\frac{a+4}{2a}},\quad
 \lambda=\bar{\lambda}\theta^{\frac{a+4}{2a}},\quad
\kappa=\bar{\kappa}\theta^{\frac{a+4}{2a}}\quad 
{\rm for}\ \ a> 0,
\end{equation}
where $\bar{\mu}$, $\bar{\lambda}$, and $\bar{\kappa}$  are constants.
In spite of this obvious physical relevance, there is no global existence result currently available beyond the small data \cite{Matsumura81energy,MN82MR784652,MN83MR713680} for the Navier--Stokes system \eqref{cns1}--\eqref{VST}  with  general adiabatic exponent $\gamma$ and transport coefficients \eqref{transport1}.

Let us mention some related results about the global existence and large-time behavior for the full compressible Navier--Stokes equations \eqref{cns1} with large data. 
The global existence and uniqueness of smooth solutions to \eqref{cns1}--\eqref{VST} in one-dimensional bounded domains are proved in the seminal work by \citet{KS77MR0468593} for \emph{constant transport coefficients} and large initial data.
The crucial step in \cite{KS77MR0468593} is to obtain the positive upper and lower bounds of the specific volume $\tau$ (i.e.\;the reciprocal of density $\rho$) and temperature $\theta$, which is achieved by means of a decent representation for  $\tau$ and the maximum principle.

The results in \cite{KS77MR0468593} have been generalized to cover the spherically and cylindrically symmetric flows.
In the case of spherical symmetry, \citet{N83MR809891} (resp.~\citet{J96MR1389908}) showed the  existence of global-in-time (generalized) solutions in bounded annular domains (resp.~in exterior domains), while \citet{CK02MR1916552} investigated the flows between a static solid core and a free boundary connected to a surrounding vacuum state.
In the cylindrically symmetric case, \citet{FS00MR1759201} established the global solvability with large data in a bounded annular domain.
Later, \citet{HJ04MR2091508} proved the global existence of a spherically or cylindrically symmetric weak solution with large discontinuous data in a ball.
As for the large-time behavior of global solutions, see \citet{J98MR1748226} (resp. \cite{J96MR1389908,Liang1405.0569}) for spherically symmetric flows in bounded domains (resp. in exterior domains), and \cite{CY15MR3279358} for cylindrically symmetric flows in bounded domains.
In all of these works the transport coefficients $\mu$, $\lambda$, and $\kappa$ are supposed to be constants.

The argument in \cite{KS77MR0468593,FS00MR1759201} can be applied to the case of \emph{constant viscosity} and temperature dependent thermal conductivity; 
see \cite{JK10MR2644363,PZ15MR3291375,W16MR3461630} for one-dimensional flows and \cite{JZ09MR2505859,QYYZ15MR3325782} for cylindrically symmetric flows. 
Under certain assumptions on  temperature dependent thermal conductivity $\kappa$,  pressure $P$,  and internal energy $e$,
\citet{WZ14MR3249721} show the global existence of spherically or cylindrically symmetric, classical and strong solutions for \eqref{cns1} with \emph{constant viscosity coefficients} $\mu$ and $\lambda$, which generalizes the work by \citet{K85MR791841} for one-dimensional flows.
The assumption of constant viscosity coefficients $\mu$ and $\lambda$ is essential in  \cite{N83MR809891,J96MR1389908,J98MR1748226,CK02MR1916552,HJ04MR2091508,Liang1405.0569,FS00MR1759201,JZ09MR2505859,CY15MR3279358,QYYZ15MR3325782,WZ14MR3249721}.

Temperature dependence of the viscosity coefficients $\mu$ and $\lambda$ turns out to have a strong influence on the solution and lead to difficulty in mathematical analysis for global solvability with large data. 
\citet{WZ16} recently obtain the first result on global well-posedness of smooth solutions to the one-dimensional Navier--Stokes system \eqref{cns1}--\eqref{VST} with general adiabatic exponent $\gamma$,  temperature dependent viscosity, and large initial data.
Unfortunately, this result do not cover the models satisfying \eqref{transport1}.
As far as we are aware, the only global solvability result currently available for \eqref{cns1}--\eqref{VST} with constitutive relations \eqref{transport1} and large data
is due to  \citet{LYZZ14MR3225502} for one-dimensional flows.
In \cite{LYZZ14MR3225502} the initial density and velocity can be large, but $\gamma-1$ and the $H^3$-norm of $\theta_0-1$ with $\theta_0$ being the initial temperature have to be small. 
In other words, this is a Nishida--Smoller type global solvability result with large data.
The original Nishida--Smoller type global solvability result is about the one-dimensional ideal polytropic isentropic compressible Euler system in \cite{NS73MR0330789} (see also \cite{Liu77MR0435618,Temple81MR626623} for the  nonisentropic case). 

It is a  natural and interesting problem to establish a Nishida--Smoller type global existence result for the spherically or cylindrically symmetric solutions of \eqref{cns1}--\eqref{bdy1} with temperature dependent transport coefficients \eqref{transport}.
Our main objective is to study this problem  and to show the exponential stability of solutions toward the constant equilibrium states as well.  
For this purpose, we reduce the initial boundary value problem \eqref{cns1}--\eqref{transport} to the corresponding problem in either spherical or cylindrical coordinates.
We let $(u, v, w)$ be the velocity components in either spherical or cylindrical coordinates.
Thus
$$\bm{u}(t,\bm{x})=\frac{\tilde{u}(t,r)}{r}\bm{x},\quad 
v=w\equiv 0$$
 in the spherically symmetric case, while
\begin{equation*}
\bm{u}(t,\bm{x})
=\frac{\tilde{u}(t,r)}{r}(x_1,x_2,0)
+\frac{\tilde{v}(t,r)}{r}(-x_2,x_1,0)+\tilde{w}(t,r)(0,0,1)
\end{equation*}
in the cylindrically symmetric case. 
In both cases, the density and temperature depend only on the time $t$ and the radius $r$, i.e. $(\rho,\theta)(t,\bm{x})=(\tilde{\rho},\tilde{\theta})(t,r).$
The system for $(\tilde\rho,\tilde{u}, \tilde{v},\tilde{w},\tilde\theta)$ then takes the form
\begin{subequations}
\label{cns2}
\begin{alignat}{2}
&\tilde\rho_t+\frac{(r^m \tilde\rho \tilde{u})_r}{r^m}=0, 
\label{cns2.a}\\
&\tilde\rho(\tilde{u}_t+\tilde{u}\tilde{u}_r)-\frac{\tilde{\rho}\tilde{v}^2}{r}+\tilde{P}_r=
\left[\frac{\tilde\nu (r^m \tilde{u})_r}{r^m}\right]_r-\frac{2m \tilde{u}\tilde\mu_r}{r},  
\label{cns2.b}\\
&\tilde\rho(\tilde{v}_t+\tilde{u}\tilde{v}_r)+\frac{\tilde{\rho}\tilde{u}\tilde{v}}{r}=
(\tilde{\mu}\tilde{v}_r)_r+\frac{2\tilde{\mu} \tilde{v}_r}{r^m}-\frac{m(\tilde{\mu}r^{m-1}\tilde{v})_r}{r^m}-\frac{\tilde{\mu}\tilde{v}}{r^{2m}}, 
\label{cns2.c}\\
&\tilde\rho(\tilde{w}_t+\tilde{u}\tilde{w}_r)=
(\tilde{\mu}\tilde{w}_r)_r+\frac{m\tilde{\mu} \tilde{w}_r}{r}, 
\label{cns2.d}\\
&\tilde\rho(\tilde{e}_t+\tilde{u}\tilde{e}_r)+\frac{\tilde{P}(r^m \tilde{u})_r}{r^m}=
\frac{(\tilde\kappa r^m\tilde\theta_r)_r}{r^m}+\tilde{\mathcal{Q}},
\label{cns2.e} 
\end{alignat}
\end{subequations}
where $\tilde\mu=\mu(\tilde\theta)$, $\tilde\nu=2\mu(\tilde\theta)+\lambda(\tilde\theta)$, $\tilde{P}=R\tilde\rho \tilde\theta$, $\tilde{e}=c_v\tilde\theta$, and
\begin{equation*}
\tilde{\mathcal{Q}}=
\frac{ \tilde\nu(r^m \tilde{u})_r^2}{r^{2m}}-\frac{2m\tilde\mu(r^{m-1} \tilde{u}^2)_r}{r^m}+\tilde{\mu}\tilde{w}_r^2+\tilde{\mu}\left[\tilde{v}_r-\frac{\tilde{v}}{r^m}\right]^2.
\end{equation*}
In spherically symmetric case, $d\in\mathbb{N}_+$, $m=d-1$, and $\tilde{v}=\tilde{w}=0$, whereas in cylindrically symmetric case, $d=3$ and $m=1$. 
The initial and boundary conditions \eqref{initial1} and \eqref{bdy1} are reduced to
\begin{alignat}{2}
\label{initial2}
(\tilde\rho,\tilde{u},\tilde{v},\tilde{w},\tilde\theta)(0,r)&=
(\tilde\rho_0,\tilde{u}_0,\tilde{v}_0,\tilde{w}_0,\tilde\theta_0)(r), \qquad && a\leq r\leq b,\\[1mm]
\label{bdy2}
(\tilde{u},\tilde{v},\tilde{w},\tilde\theta_r)(t,a)&=(\tilde{u},\tilde{v},\tilde{w},\tilde\theta_r)(t,b)=0,\qquad && t\geq 0.
\end{alignat}
These boundary conditions are supposed to be compatible with the initial data.


To establish the global existence, it is convenient to transform the initial boundary value problem \eqref{cns2}--\eqref{bdy2} into that in Lagrangian coordinates.
We introduce the Lagrangian coordinates $(t,x)$ and denote $$(\rho,u,v,w,\theta)(t,x) =(\tilde{\rho},\tilde{u},\tilde{v},\tilde{w},\tilde{\theta})(t,r),$$ 
where
\begin{equation} \label{r}
r=r(t,x)=r_0(x)+\int_0^t \tilde{u}(s,r(s,x))\mathrm{d}s,
\end{equation}
and 
\begin{equation} \label{r0}
r_0(x):=h^{-1}(x),\quad 
h(r):=\int_a^r z^m \tilde\rho_0(z)\mathrm{d}z.
\end{equation}
Notice that  the function $h$ is invertible on $[a,b]$ provided that $\tilde{\rho}_0(z)>0$ for each $z\in[a,b]$ (which will be assumed in Theorem \ref{thm}). 
Here, without loss of generality,
we set that $h(b)=1$.
Due to \eqref{cns2.a}, \eqref{bdy2}, and \eqref{r}, we see 
\begin{align} \notag
\frac{\partial}{\partial t}\int_a^{r(t,x)} z^m \tilde\rho(t,z)\mathrm{d}z
=
\frac{\partial}{\partial t}\int_b^{r(t,x)} z^m \tilde\rho(t,z)\mathrm{d}z=0.
\end{align}
Then it is easy to check that
\begin{align} \label{r_id}
\int_a^{r(t,x)} z^m \tilde\rho(t,z)\mathrm{d}z=h(r_0(x))=x
\quad{\rm and}\quad 
\int_b^{r(t,1)} z^m \tilde\rho(t,z)\mathrm{d}z=0.
\end{align}
Hence $r(t,0)=a$, $r(t,1)=b$ for $t\geq 0$, and the region $\{(t,r): t\geq 0,\ a\leq r\leq b\}$ under consideration is transformed into $\{(t,x):t\geq 0,0\leq x\leq 1\}$.
The identities \eqref{r} and \eqref{r_id} imply
\begin{align} \label{r_eq}
r_t(t,x)=u(t,x),\quad 
r_x(t,x)=r^{-m}\tau(t,x),
\end{align}
where $\tau:=1/\rho$ is the specific volume.
By virtue of \eqref{r_eq},  the system \eqref{cns2} is reformulated to that for $(\tau,u,v,w,\theta)(t,x)$ as
\begin{subequations} \label{cns}
\begin{alignat}{2}
&\tau_t=(r^m u)_x,\label{cns.a}\\ 
&u_t-\frac{v^2}{r}+r^mP_x=
r^m\left[\frac{\nu(r^m u)_x}{\tau}\right]_x-2mr^{m-1}u\mu_x,
\label{cns.b}\\ 
&v_t+\frac{uv}{r}=
r^m\left[\frac{\mu r^m v_x}{\tau}\right]_x+2\mu  v_x-m( \mu r^{m-1}  v)_x-\frac{\mu \tau v}{r^{2m}},
\label{cns.c}\\  
&w_t=
r^m\left[\frac{\mu r^m w_x}{\tau}\right]_x+m\mu  r^{m-1} w_x,
\label{cns.d}\\  
&e_t+P(r^m u)_x=
\left[\frac{\kappa r^{2m} \theta_x}{\tau}\right]_x+\mathcal{Q}, \label{cns.e}
\end{alignat}
\end{subequations}
where $t>0$, $x\in \mathcal{I}:=(0,1)$, and
\begin{align}
\label{ideal0}
P&=\frac{R \theta}{\tau},\quad\
e=c_v \theta,\quad \
c_v=\frac{R}{\gamma-1},\\ \label{Q}
\mathcal{Q}&=
\frac{\nu(r^m u)_x^2}{\tau}-2m\mu(r^{m-1} u^2)_x+\frac{\mu r^{2m} w_x^2}{\tau} +\mu \tau\left[\frac{r^m v_x}{\tau}-\frac{v}{r^m}\right]^2.
\end{align}
The initial and boundary conditions are 
\begin{alignat}{2} \label{initial0}
(\tau,u,v,w,\theta)(0,x)&=
(\tau_0,u_0,v_0,w_0,\theta_0)(x),
\qquad && x\in\widebar{\mathcal{I}},\\[1mm]               \label{bdy}
(u,v,w,\theta_x)(t,0)&=(u,v,w,\theta_x)(t,1)=0,\qquad&& t\geq 0,
\end{alignat}
where $(\tau_0,u_0,v_0,w_0,\theta_0):=
(1/\tilde{\rho}_0,\tilde{u}_0,\tilde{v}_0,\tilde{w}_0,\tilde{\theta}_0)\circ r_0,$ the symbol $\circ$ denotes composition, and $r_0$ is defined by \eqref{r0}.

We now state our main results in the following theorem.
\begin{theorem} \label{thm}
Suppose that the transport coefficients $\mu$, $\lambda$, and $\kappa$ satisfy \eqref{transport}. Let the initial data $(\tau_0,u_0,v_0,w_0,\theta_0)$ be compatible with the boundary conditions \eqref{bdy} and satisfy
\begin{align} \label{H1}
 &\|(\tau_0,u_0,v_0,w_0)\|_{H^3(\mathcal{I})}+
 \left\|\left(\sqrt{c_v}(\theta_0-1),\theta_{0xx}\right)\right\|_{H^1(\mathcal{I})}\leq \Pi_0,\\[0.5mm] \label{H2}
 &V_0^{-1}\leq \tau_0(x)\leq V_0,\qquad
 \theta_0(x)\geq V_0^{-1} \qquad {\rm for\ all}\ x\in \mathcal{I},
\end{align}
where  $\Pi_0$  and $V_0$ are positive constants independent of $\gamma-1$.
Then there exist constants $\epsilon_0>0$ and $C_1>0$, which depend only on $\Pi_0$  and $V_0$, such that if $\gamma-1\leq \epsilon_0,$ then the initial boundary value problem \eqref{cns}--\eqref{bdy} has a unique global solution  $(\tau,u,v,w,\theta)\in C([0,\infty),H^3(\mathcal{I}))$ satisfying
\begin{equation}\label{thm_C2}
C_1^{-1}\leq \tau(t,x)\leq C_1,\quad \tfrac{1}{2}\leq \theta(t,x)\leq 2
\qquad {\rm for\ all}\ (t,x)\in[0,\infty)\times\widebar{\mathcal{I}},
\end{equation}
and the exponential decay rate
\begin{equation}\label{thm_C3}
  \left\|\left(\tau-\bar{\tau},u,v,w,\theta-\bar{\theta}\right)(t)\right\|_{H^1(\mathcal{I})}+\left\|r(t)-\bar{r}\right\|_{H^2(\mathcal{I})}\leq C_{\gamma}\mathrm{e}^{-c_{\gamma}t}\qquad  {\rm for\ all}\ t\in[0,\infty),
\end{equation} 
where $C_{\gamma}$ and $c_{\gamma}$ are positive constants depending on $\gamma$, and
\begin{align}\label{hat}
\bar{\tau}=\int_{\mathcal{I}}\tau_0\mathrm{d}x,\ \  
\bar{\theta}=\int_{\mathcal{I}}\left[\theta_0+\frac{1}{2c_v}\left(u_0^2+v_0^2+w_0^2\right)\right]\mathrm{d}x,\ \ 
\bar{r}=\left[a^{m+1}+(m+1)\bar{\tau}x\right]^{\frac{1}{m+1}}.
\end{align} 
\end{theorem}
\begin{remark}
 The techniques in this paper can be applied to obtain analogous results for the initial boundary value problem \eqref{cns}--\eqref{initial0} with the following boundary conditions
 \begin{align*}
 (u,v,w)|_{x=0,1}=0,\quad
 \theta|_{x=0,1}=1,\qquad t\geq 0.
 \end{align*}  
\end{remark}
\begin{remark}
 We deduce from \eqref{H2} and \eqref{thm_C2} that no vacuum will be developed if the initial data do not contain a vacuum.
 It follows from 
 Sobolev's imbedding theorem that the unique solution constructed in Theorem \ref{thm} is a globally smooth non-vacuum solution with large initial data. 
 Moreover, this result in Lagrangian coordinates can easily be converted to an equivalent statement for the corresponding problem in Eulerian coordinates.
\end{remark}

Now we outline the main ideas to deduce Theorem \ref{thm}.
As shown in \cite{J96MR1389908,FS00MR1759201,N83MR809891},
the crucial step to construct the global solutions for the initial boundary value problem \eqref{cns}--\eqref{bdy} with large initial data is to obtain the positive upper and lower bounds of the specific volume $\tau$ and the temperature $\theta$.
In the case of constant viscosity coefficients $\mu$ and $\lambda$, the pointwise bounds for the specific volume $\tau$ and the upper bound for the temperature $\theta$ can be obtained by modifying the argument in \citet{KS77MR0468593}. 
The positive lower bound for the temperature then follows from the standard maximum principle.
However, we find that this methodology does not work in the case of density/temperature dependent viscosity.

For ideal polytropic gases \eqref{cns}--\eqref{ideal0}, the temperature $\theta$ satisfies
\begin{equation*}
c_v\theta_t+\frac{R\theta(r^m u)_x}{\tau}=
\left[\frac{\kappa r^{2m} \theta_x}{\tau}\right]_x+\mathcal{Q}.
\end{equation*}
It can be expected to get the uniform-in-time bounds for $\|\sqrt{c_v}(\theta-1,\theta_t)(t)\|_{H^1(\mathcal{I})}$ by very careful energy estimates even when the viscosity and thermal conductivity coefficients are functions of the temperature.
Recalling that $c_v=R/(\gamma-1)$, under the \emph{a priori} assumption that $\gamma-1$ is sufficiently small, we can use the smallness of $\|(\theta-1,\theta_t)(t)\|_{H^1(\mathcal{I})}$ to handle the possible growth of solutions induced by the temperature dependence of the viscosity. 
The bounds for the specific volume $\tau$ from below and above can be established by developing the argument by \citet{K68MR0227619} (see Lemma~\ref{lem_kanel}).

There are three main differences between our results for the symmetric flows \eqref{cns} and the work in \cite{LYZZ14MR3225502} restricted to the one-dimensional case.
  \begin{list}{}{\setlength{\parsep}{\parskip}
  		\setlength{\itemsep}{0.1em}
  		\setlength{\labelwidth}{2em}
  		\setlength{\labelsep}{0.4em}
  		\setlength{\leftmargin}{2.2em}
  		\setlength{\topsep}{1mm}
  	}
  	\item[1.]
First, the system \eqref{cns} under consideration here is more complicated than the one-dimensional compressible Navier--Stokes system (i.e.\;the system \eqref{cns} with $m=v=w=0$) which makes the form and treatment of equations simpler.
\item[2.]
The second difference is that the analysis in \cite{LYZZ14MR3225502} takes place in the whole space without boundary. Thus the energy estimates for the two-order and three-order derivatives of the solutions ($\|u_{xx}(t)\|_{L^2(\mathbb{R})}$ etc.) can be deduced directly from integrating by parts (see, for instance, \cite[Page\;2205,\;(3.41)]{LYZZ14MR3225502} for the estimate of $\|u_{xx}(t)\|_{L^2(\mathbb{R})}$).
Due to the presence of the boundary conditions \eqref{bdy} in our case, we cannot use the methods in \cite{LYZZ14MR3225502} to estimate the two-order and three-order derivatives of the solutions $(\tau,u,v,w,\theta)$. 
To overcome such a difficulty, we make the estimates for $(u_t,v_t,w_t,\theta_t)$ and $(u_{xt},v_{xt},w_{xt},\theta_{xt})$, which yield the  bounds for $\|(\tau_{xx},u_{xx},v_{xx},w_{xx},\theta_{xx})(t)\|_{H^1(\mathcal{I})}$.
\item[3.]
A third difference concerns the assumptions on the initial data. 
In \citet{LYZZ14MR3225502}, $\|\sqrt{c_v}(\theta_{0}-1)\|_{H^3(\mathbb{R})}$ is required to be bounded by some $(\gamma-1)$-independent positive constant. 
According to this assumption, the $H^1$-norm of $\theta_{0xx}$ has to be small because of the smallness of $\gamma-1$. 
In Theorem \ref{thm}, $\|\theta_{0xx}\|_{H^1(\mathcal{I})}$ is assumed in \eqref{H1} to be bounded by some $(\gamma-1)$-independent positive constant. Thus $\|\theta_{0xx}\|_{H^1(\mathcal{I})}$ can be large even when the adiabatic exponent $\gamma$ goes to $1$. 
\end{list}

The rest of this paper is organized as follows. 
First, in Section \ref{sec2}, we derive a number of desired \emph{a priori} estimates. More specifically, the basic energy estimate is obtained in Subsection \ref{subsec1};
the uniform-in-time pointwise bounds of the specific volume $\tau$ is shown in Subsection \ref{subsec2} by applying the argument developed by Kanel$'$; the estimates on first-order, second-order, and third-order derivatives of the solution $(\tau,u,v,w,\theta)$ will be deduced in Subsections \ref{subsec3}, \ref{subsec4}, and \ref{subsec5}, respectively. Finally, in Section 3, by combining the a priori estimates and the continuation argument, we prove the existence, uniqueness and exponential decay rate of global-in-time solutions for the problem \eqref{cns}--\eqref{bdy}.
 
 \bigbreak
 \noindent{\bf Notations}.
Throughout this paper, we use $\mathcal{I}:=(0,1)$. 
 For  $1\leq q\leq \infty$ and $k\in \mathbb{N}$, we denote by $L^q(\mathcal{I})$ the usual Lebesgue space on $\mathcal{I}$ equipped with the norm $\|{\cdot}\|_{L^q(\mathcal{I})}$ and by $H^k(\mathcal{I})$ the standard Sobolev space in the $L^2$ sense equipped with the norm $\|{\cdot}\|_{H^k(\mathcal{I})}$.
 For notational simplicity, we shall use
 $$\|{\cdot}\|:=\|{\cdot}\|_{L^2(\mathcal{I})},\quad
 \|{\cdot}\|_k:=\|{\cdot}\|_{H^k(\mathcal{I})},\quad \|{\cdot}\|_{L^q}:=\|{\cdot}\|_{L^q(\mathcal{I})}
 .$$
 If $I\subset\mathbb{R}$ and $X$ is a Banach space, then
 we denote by $C(I; X)$ the space of continuous functions on $I$ with values in $X$, by $L^q(I; X)$ the space of $L^q$-functions on $I$ with values in $X$, and by $\|{\cdot}\|_{L^q(I; X)}$  the norm of the space $L^q(I; X)$.
To simplify the presentation, we employ $C$, $c$, and $C_i$ ($i\in\mathbb{N}$) to denote various positive constants, depending only on  $\Pi_0$  and $V_0$, where  $\Pi_0$  and $V_0$ are determined by \eqref{H1}--\eqref{H2}. 
Hence $C$, $c$, and $C_i$ are independent of $\gamma-1$ and $t$.
The symbol $A\lesssim B$ (or $B\gtrsim A$) means that $A\leq C B$ holds uniformly for some $(\gamma-1)$-independent constant $C$.
We also use $C_{\gamma}$, $c_{\gamma}$, and $C_i(\gamma)$ ($i\in\mathbb{N}_+$) to denote positive constants depending on $\gamma$.

\section{A priori estimates} \label{sec2}
This section is devoted to deriving certain a priori estimates on the solutions $(\tau, u, v, w, \theta)\in X(0,T;M,N)$ to the initial boundary value problem \eqref{cns}--\eqref{bdy} with temperature dependent transport coefficients \eqref{transport} for $T>0$, $M\geq 1$, and  $N\geq 1$.
Here we define the set
\begin{equation*}
\begin{split}
&X(t_1,t_2 ;M,N):=\Big\{
(\tau,u,v,w,\theta)\in C([t_1,t_2];H^3(\mathcal{I})):\\[0.5mm]
&\quad\quad \tau_x\in L^2(t_1,t_2;H^2(\mathcal{I})),\
\tau_t\in C([t_1,t_2];H^2(\mathcal{I}))\cap L^2(t_1,t_2;H^2(\mathcal{I})),\\[0.5mm]
&\quad\quad   (u_t,v_t,w_t,\theta_t)\in C([t_1,t_2];H^1(\mathcal{I}))\cap L^2(t_1,t_2;H^2(\mathcal{I})),\
\mathcal{E}_{t_1}(t_2)\leq N^2,\\
&\quad\quad (u_x,v_x,w_x,\theta_x)\in L^2(t_1,t_2;H^3(\mathcal{I})),\
 \tau(t,x)\geq M^{-1} \ \forall\ (t,x)\in[t_1,t_2 ]\times \mathcal{I}
\Big\},  
\end{split}
\end{equation*}
for constants $M$, $N$, $t_1$, and $t_2$ ($t_1\leq t_2$),  where
\begin{align} \label{E}
\mathcal{E}_{t_1}(t_2):=\sup_{t\in[t_1,t_2]}\left\{
\|(u,\sqrt{c_v}(\theta-1),\sqrt{c_v}\theta_t)(t)\|_1^2
+\|\theta_{xx}(t)\|^2\right\}
+\int_{t_1}^{t_2}\left\|\sqrt{c_v}\theta_t(s)\right\|_1^2\mathrm{d}s,
\end{align}
with
\begin{align}
\label{initial_thet}
\theta_t|_{t=t_1}&:=\left.\frac{1}{c_v}\left[-P(r^m u)_x+\left[\frac{\kappa r^{2m} \theta_x}{\tau}\right]_x+\mathcal{Q}\right]\right|_{t=t_1},\\
\label{initial_thext}
\theta_{xt}|_{t=t_1}&:=\left.\frac{1}{c_v}\left[-P(r^m u)_x+\left[\frac{\kappa r^{2m} \theta_x}{\tau}\right]_x+\mathcal{Q}\right]_x\right|_{t=t_1}.
\end{align}

We set without loss of generality that $R=1$ and hence $c_v=1/(\gamma-1)$.
For the sake of simplicity, we will use the following abbreviation:
$$
\lrn\cdot\rrn:=\|{\cdot}\|_{L^\infty([0,T];L^{\infty}( \mathcal{I}))}.$$ 
Since $(\tau, u, v, w, \theta)\in X(0,T;M,N)$, it follows from  Sobolev's inequality that
\begin{align}
\label{apriori0}
&M^{-1}\leq \tau(t,x)\lesssim N,\quad
\|(\theta-1,\theta_t)(t)\|_1\lesssim  (\gamma-1)^{\frac{1}{2}}N\quad
\forall\ (t,x)\in[0,T]\times\mathcal{I},
\\ \label{apriori1}
&\lrn \theta_x \rrn\lesssim (\gamma-1)^{\frac{1}{4}}N,\quad
\lrn (\theta-1,\theta_t) \rrn \lesssim (\gamma-1)^{\frac{1}{2}}N,\quad
\int_{0}^{T}\|\theta_{t}(t)\|_1^2\mathrm{d}t\lesssim (\gamma-1)N^2.
\end{align}

We shall make repeated use of the following estimate:
 \begin{align}   \label{E_r}
 a\leq r(t,x)\leq b\quad
 \forall\ (t,x)\in[0,T]\times\mathcal{I},
 \end{align}
 which follows from \eqref{r_id} and \eqref{r_eq}.

\subsection{Basic energy estimate} \label{subsec1}
In the following lemma we show the basic energy estimate and the pointwise bounds for $\theta$.
\begin{lemma} \label{lem_bas}
 Assume that the conditions listed in Theorem \ref{thm} hold.
 Then there exists a positive constant $\epsilon_1$ depending only on  $\Pi_0$  and $V_0$, such that if
 \begin{equation}\label{apriori2}
 (\gamma-1)^{\frac{1}{4}} M^2 N^2\leq \epsilon_1,
 \end{equation}
 then 
 \begin{align}   \label{E_theta}
 & \tfrac{1}{2}\leq \theta(t,x)\leq 2\qquad
 \forall\ (t,x)\in[0,T]\times\mathcal{I},\\[1.5mm] \label{E_bas1} 
 &\sup_{t\in[0,T]}\left\|\left(\sqrt{\phi\left(\frac{\tau}{\bar{\tau}}\right)},u,v,w,\sqrt{c_v}(\theta-1)\right)(t)\right\|^2\lesssim 1,\\ \label{E_bas2}
 &\int_0^T\int_{\mathcal{I}}\left[m\tau u^2+\tau v^2+
 \frac{\tau_t^2+u_x^2+v_x^2+w_x^2+\theta_x^2}{\tau}\right]\mathrm{d}x\mathrm{d} t \lesssim 1,
 \end{align}
 where $\phi(z):=z-\ln z -1$ and $\bar{\tau}$ is given by \eqref{hat}.
\end{lemma}
\begin{proof}
The estimate \eqref{E_theta} follows immediately from \eqref{apriori1}.
 In order to prove \eqref{E_bas1} and \eqref{E_bas2}, we deduce an entropy-type energy estimate for the initial boundary value problem \eqref{cns}--\eqref{bdy}.
 Let $\hat{\theta}>0$ be an arbitrary but fixed constant.
 Multiplying \eqref{cns.a}, \eqref{cns.b}, \eqref{cns.c}, \eqref{cns.d}, and \eqref{cns.e} by $\hat{\theta}\left(\bar{\tau}^{-1}-\tau^{-1}\right)$, $u$, $v$, $w$, and $\big(1-\hat{\theta}\theta^{-1}\big)$, respectively, we have
 \begin{align}
 \eta_{\hat{\theta}}(\tau,u,v,w,\theta)_t
 +\frac{\hat{\theta}\kappa r^{2m}\theta_x^2}{\tau\theta^2}
 +\frac{\hat{\theta}\mathcal{Q}}{\theta}=\mathcal{R}_x, \label{id1.1}
 \end{align}
 where 
 \begin{align}
 \label{eta}
 \eta_{\hat{\theta}}(\tau,u,v,w,\theta):=\hat{\theta}\phi\left(\frac{\tau}{\bar{\tau}}\right)+\tfrac{1}{2}
 \left(u^2+v^2+w^2\right)+c_v\hat{\theta}\phi\left(\frac{\theta}{\hat{\theta}}\right)
 \end{align}
 is the relative entropy (see, for instance, \cite{NYZ04MR2083790}), $\mathcal{Q}$ is given by \eqref{Q}, and
 \begin{align} \notag
 \mathcal{R}:=~&\frac{\nu r^m u\left(r^m u\right)_x}{\tau}+
 \left[1-\frac{\hat{\theta}}{\theta}\right]
 \frac{\kappa r^{2m}\theta_x}{\tau}
 + r^m u\left[\frac{\hat{\theta}}{\bar{\tau}}-\frac{\theta}{\tau}\right]\\ \notag
 &-2m\mu r^{m-1} u^2
 +\frac{\mu r^{2m} (vv_x+ww_x)}{\tau}-m \mu r^{m-1} v^2.
 \end{align}
 Thanks to the boundary conditions \eqref{bdy}, we integrate \eqref{id1.1} over $\mathcal{I}$ to obtain
  \begin{align}
  \frac{\mathrm{d}}{\mathrm{d}t}\int_{\mathcal{I}} \eta_{\hat{\theta}}(\tau,u,v,w,\theta)
  +\int_{\mathcal{I}}\hat{\theta}\left[\frac{\kappa r^{2m}\theta_x^2}{\tau\theta^2}
  +\frac{\mathcal{Q}}{\theta}\right]=0.  \label{id_bas}
  \end{align}
  By virtue of \eqref{r_eq}, we infer that, for $\sigma>0$, 
  \begin{align*}
  &\nu\left(r^m u\right)_x^2-
  2m\mu \tau\left(r^{m-1} u^2\right)_x =\nu r^{2m} u_x^2+(2m\mu+
  m^2 \lambda)\frac{\tau^2 u^2}{r^2}+2m \lambda r^{m-1}\tau u u_x\\
  &\quad =\left[\frac{m\lambda r^m u_x}{\sqrt{\sigma}}+
  \frac{\sqrt{\sigma} \tau u}{r}\right]^2+
  \left[\nu-\frac{m^2\lambda^2}{\sigma}\right]r^{2m}u_x^2+
  \left[2m\mu+m^2\lambda -\sigma\right]\frac{\tau^2u^2}{r^2}\\
  &\quad \geq \left[\nu-\frac{m^2\lambda^2}{\sigma}\right]r^{2m}u_x^2
  +\left[2m\mu+m^2\lambda -\sigma\right]\frac{\tau^2u^2}{r^2}.
  \end{align*}
In light of \eqref{transport} and \eqref{E_theta}, 
 we can choose $\sigma =m(m\lambda^2+2\mu^2+(m+1)\mu\lambda)/\nu>0$ so that
\begin{equation} \label{est1.1}
\frac{\nu\left(r^m u\right)_x^2-2m\mu \tau\left(r^{m-1} u^2\right)_x}{\tau\theta}\geq \frac{\beta}{\tau}\left[r^{2m}u_x^2+\frac{m \tau^2 u^2}{r^2}\right],
\end{equation}
where $\beta$ is given by 
\begin{align*} 
&\beta:=\frac{2\mu+(m+1)\lambda}{\theta}\min\left\{\frac{\mu \nu}{m\lambda^2+2\mu^2+(m+1)\mu\lambda},\frac{\mu}{\nu}\right\}.
\end{align*}
Using \eqref{bdy}, \eqref{apriori0}, \eqref{apriori1}, and \eqref{E_theta}, we integrate by parts to find
\begin{equation*} 
2\int_{\mathcal{I}} \frac{\mu(\theta)}{\theta}vv_x
=-\int_{0}^{T}\int_{\mathcal{I}}
\frac{\mathrm{d}}{\mathrm{d}\theta}\left[\frac{\mu(\theta)}{\theta}\right]\theta_xv^2
\lesssim (\gamma-1)^{\frac{1}{4}}NM\int_{\mathcal{I}} \tau v^2,
\end{equation*} 
which combined with \eqref{apriori2} and \eqref{E_theta} implies
\begin{align}\label{est1.2}
\int_{\mathcal{I}}\frac{\mu \tau}{\theta}\left[\frac{r^m v_x}{\tau}-\frac{v}{r^m}\right]^2\gtrsim 
\int_{\mathcal{I}}\frac{v_x^2}{\tau}+(1-\epsilon_1)\int_{\mathcal{I}}\tau v^2.
\end{align}
Plug \eqref{est1.1} and \eqref{est1.2} into \eqref{id_bas}, 
and use \eqref{E_theta} to derive that if \eqref{apriori2} holds for some sufficiently small $\epsilon_1>0$, then
\begin{align}
\frac{\mathrm{d}}{\mathrm{d}t}\int_{\mathcal{I}} \eta_{\hat{\theta}}(\tau,u,v,w,\theta)
+c\int_{\mathcal{I}}\hat{\theta}\left[m\tau u^2+\tau v^2+
\frac{u_x^2+v_x^2+w_x^2+\theta_x^2}{\tau}\right]
\leq 0.  \label{est1}
\end{align} 
It follows from \eqref{E_theta}, \eqref{cns.a}, and \eqref{r_eq} that 
\begin{align*}
\phi(\theta)\gtrsim (\theta-1)^2\quad{\rm and}\quad
\frac{\tau_t^2}{\tau}\lesssim m\tau u^2+\frac{u_x^2}{\tau}.
\end{align*}
If we take $\hat{\theta}=1$, by virtue of the conditions assumed in Theorem \ref{thm}, we infer
\begin{align*}
\int_{\mathcal{I}} \eta_{1}(\tau_0,u_0,v_0,w_0,\theta_0)\lesssim 1.
\end{align*}
Integrating \eqref{est1} with $\hat{\theta}=1$ over $(0,t)$ yields the estimates \eqref{E_bas1} and \eqref{E_bas2}.
\end{proof}
 
\subsection{Pointwise bounds for the specific volume} \label{subsec2}
In this subsection we employ the argument developed by Kanel$'$ \cite{K68MR0227619,LYZZ14MR3225502} to obtain the uniform bounds for the specific volume $\tau$.
To this end, we first make the estimate for $\|{\tau_x}/{\tau}\|$  in the following lemma.  
\begin{lemma} \label{lem_taux}
If the conditions listed in Lemma \ref{lem_bas} hold for a sufficiently small $\epsilon_1$, then 
\begin{equation} \label{E_taux}
\sup_{t\in[0,T]}\left\|\frac{\tau_x}{\tau}(t)\right\|^2
+\int_0^T\int_{\mathcal{I}}\frac{\tau_x^2}{\tau^3}
\lesssim 1+\lrn\ln \tau\rrn.
\end{equation}
\end{lemma}
\begin{proof}
According to the chain rule, we have
\begin{equation}
\label{id2.1}
\left[\frac{\nu \tau_x}{\tau}\right]_t=
\left[\frac{\nu \tau_t}{\tau}\right]_x+
\frac{\nu'(\theta)}{\tau}\left(\tau_x\theta_t-\tau_t\theta_x\right),
\end{equation} 
which combined with \eqref{cns.a} and \eqref{cns.b} implies
\begin{align} 
\left[\frac{\nu \tau_x}{\tau}\right]_t-\frac{u_t}{r^{m}}-P_x
=-\frac{v^2}{r^{m+1}}+\frac{2m u\mu_x}{r}+
\frac{\nu'(\theta)}{\tau}\left(\tau_x\theta_t-\tau_t\theta_x\right).
\label{id2.2}
\end{align}
Multiply \eqref{id2.2} by $\nu \tau_x/\tau$ and use \eqref{cns.a}, \eqref{id2.1} to deduce
\begin{align*} 
&\left[\frac{1}{2}\left(\frac{\nu \tau_x}{\tau}\right)^2-
\frac{u}{r^{m}}\frac{\nu \tau_x}{\tau}\right]_t
+\frac{\nu \theta\tau_x^2}{\tau^3}
+\left[\frac{u}{r^{m}}\frac{\nu \tau_t}{\tau}\right]_x-\left(\frac{u}{r^m}\right)_x \frac{\nu \tau_t}{\tau}
-\frac{\nu \tau_x\theta_x}{\tau^2}\\[0.5mm]
&\quad =\frac{\nu \tau_x}{\tau}\frac{m u^2-v^2}{r^{m+1}}
+\frac{2m u\mu_x}{r}\frac{\nu \tau_x}{\tau}
+\frac{\nu'(\theta)}{\tau^2}\left(\nu\tau_x-r^{-m}u\tau\right)
\left(\tau_x\theta_t-\tau_t\theta_x\right).
\end{align*}
Integrating the last identity, we get
\begin{align}
\frac{\mathrm{d}}{\mathrm{d}t}\int_{\mathcal{I}}
\left[\frac{1}{2}\left(\frac{\nu \tau_x}{\tau}\right)^2-
\frac{u}{r^{m}}\frac{\nu \tau_x}{\tau}\right]
+\int_{\mathcal{I}} \frac{\nu \theta \tau_x^2}{\tau^3}
=\sum_{q=1}^{7}\mathcal{K}_q, 
\label{id_taux}
\end{align}
where each term $\mathcal{K}_q$ in the decomposition will be defined and estimated below. First we consider the term
\begin{align*}
\mathcal{K}_1:=
\int_{\mathcal{I}}\left(\frac{u}{r^m}\right)_x\frac{\nu \tau_t}{\tau}.
\end{align*}
According to \eqref{r_eq} and \eqref{E_theta}, we infer
\begin{align}
\mathcal{K}_1=
\int_{\mathcal{I}} \frac{\nu \tau_t}{\tau} 
\left(\frac{u_x}{r^m}-\frac{m\tau u}{r^{2m+1}}\right)\lesssim 
\int_{\mathcal{I}}\left[\frac{u_x^2}{\tau}+
\frac{\tau_t^2}{\tau}+m\tau u^2\right].
\label{K1}
\end{align}
The second term
is defined and estimated as
\begin{align} 
\mathcal{K}_2:=\int_{\mathcal{I}}\frac{\nu \tau_x\theta_x}{\tau^2}\leq
\epsilon \int_{\mathcal{I}} \frac{ \tau_x^2}{\tau^3}
+C(\epsilon) \int_{\mathcal{I}}\frac{\theta_x^2}{\tau }.
\label{K2}
\end{align}
For the term 
\begin{align*}
\mathcal{K}_3:=
\int_{\mathcal{I}}\frac{\nu  \tau_x}{\tau}\frac{m u^2-v^2}{r^{m+1}},
\end{align*}
by virtue of the boundary conditions \eqref{bdy},
we integrate by parts and use \eqref{r_eq}, \eqref{E_theta} to derive
\begin{align*}
\mathcal{K}_3=-\int_{\mathcal{I}} 
\ln\tau \left[\frac{\nu}{r^{m+1}}(m u^2-v^2)\right]_x\lesssim
\lrn \ln \tau\rrn
\int_{\mathcal{I}} \left|\left(m\theta_x u^2,\theta_x v^2,
m\tau u^2,\tau v^2,mu u_x,vv_x\right)\right|.
\end{align*}
In view of \eqref{apriori0} and \eqref{apriori1}, we have
\begin{align}
\mathcal{K}_3\lesssim \lrn \ln \tau\rrn\left[1+
(\gamma-1)^{\frac{1}{4}}NM\right]
\int_{\mathcal{I}}\left[m\tau u^2+\tau v^2
+\frac{u_x^2+v_x^2}{\tau}\right].
\label{K3}
\end{align}
The term 
\begin{align*}
\mathcal{K}_4=
\int_{\mathcal{I}}\frac{\nu'(\theta)}{\tau}r^{-m}u \tau_t\theta_x
\end{align*}
can be treated by using \eqref{apriori0} as
\begin{align}\notag
\mathcal{K}_4&\lesssim 
\|u\|^{\frac{1}{2}}\|u_x\|^{\frac{1}{2}}\left\|\frac{\tau_t}{\sqrt{\tau}}\right\|
\left\|\frac{\theta_x}{\sqrt{\tau}}\right\|
\\  \notag &\lesssim 
\left\|\frac{\tau_t}{\sqrt{\tau}}\right\|^2+
(\gamma-1)^{\frac{1}{2}}N^4M^2
\left\|\frac{u_x}{\sqrt{\tau}}\right\|\left\|\frac{\theta_x}{\sqrt{\tau}}\right\|\\
&\lesssim \left[1+
(\gamma-1)^{\frac{1}{2}}N^4M^2\right]
\int_{\mathcal{I}}\frac{\tau_t^2+u_x^2+\theta_x^2}{\tau}.
\label{K4}
\end{align}
We have from \eqref{apriori1} and \eqref{apriori2} that
\begin{align}
\mathcal{K}_5:=\int_{\mathcal{I}} \frac{\nu'(\theta)\nu}{\tau^2} \theta_t\tau_x^2\lesssim 
N\lrn\theta_t\rrn\int_{\mathcal{I}}\frac{\tau_x^2}{\tau^3}
\lesssim \epsilon_1 \int_{\mathcal{I}}\frac{\tau_x^2}{\tau^3}.
\label{K5}
\end{align}
Using \eqref{apriori1}, we get
\begin{align}\notag
\mathcal{K}_6:=～&\int_{\mathcal{I}} 
\frac{\tau_x}{\tau}\theta_x \left(2m r^{-1}u\mu'(\theta)\nu 
-\nu \nu'(\theta)\tau^{-1}\tau_t\right)\\ 
\lesssim ～&
\epsilon\int_{\mathcal{I}}\frac{\tau_x^2}{\tau^3}
+C(\epsilon)(\gamma-1)^{\frac{1}{2}}N^2M^2\int_{\mathcal{I}}\left[
m\tau u^2+\frac{\tau_t^2}{\tau}\right].
\label{K6}
\end{align}
For the last term
\begin{align*}
\mathcal{K}_7:=-\int_{\mathcal{I}} \frac{\nu'(\theta)}{\tau} r^{-m}u \tau_x\theta_t ,
\end{align*}
we have from \eqref{apriori1} that
\begin{align}
\mathcal{K}_7
\lesssim \epsilon\int_{\mathcal{I}}\frac{\tau_x^2}{\tau^3}
+C(\epsilon)\lrn u\rrn^2 N\|\theta_t\|^2
\lesssim \epsilon\int_{\mathcal{I}}\frac{\tau_x^2}{\tau^3}
+C(\epsilon) N^3  \|\theta_t\|^2.
\label{K7}
\end{align}
Plugging \eqref{K1}--\eqref{K7} into \eqref{id_taux} and using \eqref{apriori2}, \eqref{E_theta} yield
\begin{align} \notag
&\frac{\mathrm{d}}{\mathrm{d}t}\int_{\mathcal{I}}
\left[\frac{1}{2}\left(\frac{\nu \tau_x}{\tau}\right)^2-
\frac{u}{r^{m}}\frac{\nu \tau_x}{\tau}\right]
+\int_{\mathcal{I}} \frac{\tau_x^2}{\tau^3}\\
&\quad\lesssim N^3  \|\theta_t\|^2+
\left[1+ \lrn \ln \tau\rrn\right]
\int_{\mathcal{I}}\left[m\tau u^2+v^2+
\frac{\tau_t^2+u_x^2+v_x^2+\theta_x^2}{\tau}\right]. \label{est2}
\end{align}
We integrate \eqref{est2} over $(0,t)$, apply Cauchy's inequality, and use \eqref{apriori1}, \eqref{apriori2}, \eqref{E_bas2} to conclude \eqref{E_taux}.
\end{proof}

Now we establish the uniform bounds for the specific volume, which are essential for the proof of the main theorem.
\begin{lemma} \label{lem_kanel}
If the conditions listed in Lemma \ref{lem_bas} hold for a sufficiently small $\epsilon_1$, then 
\begin{equation} \label{E_tau}
C_1^{-1}\leq \tau(t,x)\leq C_1\qquad{\rm for\ all}\ (t,x)\in[0,T]\times\mathcal{I}.
\end{equation}
\end{lemma} 
\begin{proof}
Define
$$\Phi(\tau):=\int_{1}^{\tau/\bar{\tau}}\frac{\sqrt{\phi(z)}}{z}\mathrm{d}z,$$
where $\bar{\tau}$ is defined by \eqref{hat}.
We infer that for suitably large constant $C$ and $\tau\geq C$,
\begin{equation*}
\Phi(\tau) \gtrsim \int_{C}^{\tau/\bar{\tau}}\frac{\sqrt{\phi(z)}}{z}\mathrm{d}z\gtrsim\int_C^{\tau/\bar{\tau}} z^{-\frac12}\mathrm{d}z.
\end{equation*}
Hence 
$$\tau^{\frac12}\lesssim 1+|\Phi(v)|\qquad \forall\ \tau\in(0,\infty). $$
Similarly, it follows that
$$|\ln \tau|^{\frac{3}{2}}\lesssim 1+|\Phi(\tau)|\qquad \forall\ 
\tau\in(0,\infty).$$
Thus, we have
\begin{equation}\label{kanel1}
\lrn \tau\rrn^{\frac12}+\lrn\ln \tau\rrn^{\frac{3}{2}}\lesssim
 1+\sup_{(t,x)\in[0,T]\times\mathcal{I}}|\Phi(\tau(t,x))|.
\end{equation}
By virtue of \eqref{cns.a} and \eqref{bdy}, we obtain
\begin{equation} \label{id_intau}
\int_{\mathcal{I}} \tau(t,x)\mathrm{d}x=\bar{\tau}\qquad \forall\ t\in[0,T],
\end{equation}
which implies that for each $t\in[0,T]$, there exists $y(t)\in \mathcal{I}$ such that $\tau(t,y(t))=\bar{\tau}$. Hence
\begin{align*}\label{kanel2}
&|\Phi(\tau)(t,x)|=
\left|\int_{y(t)}^x\frac{\partial}{\partial x}\Phi(\tau(t,y))\mathrm{d}y\right|\\
&\quad \leq  \int_{\mathcal{I}}\sqrt{\phi\left(\frac{\tau}{\bar{\tau}}(t,y)\right)}\left|\frac{\tau_x}{\tau}(t,y)\right|\mathrm{d}y
\leq\left\|\sqrt{\phi\left(\frac{\tau}{\bar{\tau}}\right)}(t)\right\|\left\|\frac{\tau_x}{\tau}(t)\right\|,
\end{align*}
which combined with Lemmas \ref{lem_bas}--\ref{lem_taux}  yields
\begin{equation}\label{kanel3}
|\Phi(\tau)(t,x)| \lesssim 1+\lrn\ln \tau \rrn^{\frac12}\qquad 
\forall\ (t,x)\in[0,T]\times\mathcal{I}.
\end{equation}
Combine \eqref{kanel1} with \eqref{kanel3} to deduce \eqref{E_tau}.
\end{proof}

We plug \eqref{E_tau} into \eqref{E_bas1}, \eqref{E_bas2}, and \eqref{E_taux} to obtain the following corollary.
\begin{corollary} \label{cor_1}
If the conditions listed in Lemma \ref{lem_bas} hold for a sufficiently small $\epsilon_1$, then 
\begin{align} 
\sup_{t\in[0,T]}\left\| \left(\tau,\tau_x,u,v,w,\sqrt{c_v}(\theta-1)\right)(t)\right\|^2
+\int_0^T\left\|\left(\tau_x,\tau_t,\sqrt{m} u, u_x,v, v_x, w_x,\theta_x \right)(t)\right\|^2\mathrm{d} t \leq C_2.
\label{E_1a}
\end{align} 
\end{corollary}

\subsection{Estimates on first-order derivatives}
\label{subsec3}
This part is devoted to deducing the bounds for  $\|(u_x,v_x,w_x,\sqrt{c_v}\theta_x)(t)\|$ uniformly in time $t$.
We first make the estimate for $w_x$ in the next lemma.
\begin{lemma} \label{lem_wx}
If the conditions listed in Lemma \ref{lem_bas} hold for a sufficiently small $\epsilon_1$, then 
\begin{equation} \label{E_wx}
 \sup_{t\in[0,T]}\|w_x(t)\|^2 +\int_{0}^{T}\|w_{xx}(t)\|^2\mathrm{d}t\leq C_3.
\end{equation}
\end{lemma}
\begin{proof}
Multiply \eqref{cns.d} by $w_{xx}$, integrate the resulting identity over $\mathcal{I}$, and use $$w_t(t,0)=w_t(t,1)=0$$ to find
\begin{align}
\frac{\mathrm{d}}{\mathrm{d}t}\|w_x(t)\|^2
+2\int_{\mathcal{I}} \frac{\mu r^{2m}w_{xx}^2}{\tau}
=-2\int_{\mathcal{I}}w_{xx}r^mw_{x}\left[\frac{\mu r^m}{\tau}\right]_x
-2m\int_{\mathcal{I}}w_{xx}\mu  r^{m-1}w_x.
\notag  
\end{align}
In view of \eqref{r_eq}, \eqref{E_theta} and \eqref{E_tau},
we have
\begin{align*}
\left|\left[\frac{\mu r^m}{\tau}\right]_x\right|\lesssim 
\left|\left(1,\theta_x,\tau_x\right)\right|.
\end{align*}
By virtue of \eqref{apriori0}, \eqref{apriori1} and \eqref{apriori2},  we have
\begin{align}\label{apriori5}
\sup_{t\in[0,T]}\|(\theta-1,\theta_t)(t)\|_1+\lrn(\theta_x,\theta_t)\rrn+
\int_{0}^{T}\|\theta_{t}(t)\|_1^2\mathrm{d}t\lesssim 1.
\end{align}
Then using Cauchy's and Sobolev's inequalities yields
\begin{align}\notag
&\frac{\mathrm{d}}{\mathrm{d}t}\|w_x(t)\|^2
+\| w_{xx}(t)\|^2\lesssim \int_{\mathcal{I}} w_x^2(1+\theta_x^2+\tau_x^2)\\ \notag
&\quad\lesssim \left[1+\lrn\theta_x\rrn^2\right]\|w_x(t)\|^2
+\sup_{s\in[0,T]}\|\tau_x(s)\|^2\|w_x(t)\|_{L^{\infty}}^2\\ \notag
&\quad\lesssim 
\left[1+C(\epsilon)\sup_{s\in[0,T]}\|\tau_x(s)\|^4\right]\|w_x(t)\|^2+\epsilon \|w_{xx}(t)\|^2,
\end{align}
which combined with \eqref{E_1a} implies
\begin{align} 
\frac{\mathrm{d}}{\mathrm{d}t}\|w_x(t)\|^2+\| w_{xx}(t)\|^2\lesssim \|w_x(t)\|^2.\label{est3}
\end{align}
We conclude \eqref{E_wx} by integrating \eqref{est3} over $(0,t)$. 
\end{proof}

In the following lemma we obtain the bounds for $\|u_x(t)\|$ and $\|v_x(t)\|$ uniformly in time $t$. 
\begin{lemma}
If the conditions listed in Lemma \ref{lem_bas} hold for a sufficiently small $\epsilon_1$, then 
\begin{equation} \label{E_uxvx}
\sup_{t\in[0,T]}\|(u_x,v_x)(t)\|^2 +\int_{0}^{T}\|(u_{xx},v_{xx})(t)\|^2\mathrm{d}t\leq C_4.
\end{equation}
\end{lemma}
\begin{proof}
 Multiply \eqref{cns.b} and \eqref{cns.c} by $u_{xx}$ and $v_{xx}$, respectively, and then add the resulting identities to get
 \begin{align*}
    &\left(u_x^2+v_x^2\right)_t-2\left(u_xu_t+v_xv_t\right)_x
    +\frac{2r^{2m}}{\tau}\left(\nu u_{xx}^2+\mu v_{xx}^2\right)
    -2u_{xx}r^m P_x\\
   ~& =2\left[-u_{xx}\frac{v^2}{r}+v_{xx}\frac{uv}{r}\right]
    -2u_{xx} r^m\left[\left(\frac{\nu (r^m u)_x}{\tau}\right)_x
    -\frac{\nu r^m u_{xx}}{\tau}\right]+4m r^{m-1}u \mu_x u_{xx}\\
    &\quad -2v_{xx} \left[r^m v_x \left(\frac{\mu r^m}{\tau}\right)_x+2\mu r^{m-1} v_x\
    -m(\mu r^{m-1} v)_x-\frac{\mu\tau v }{r^{2m}}\right].
 \end{align*}
 Integrating the last identity and using the boundary conditions
 $$u_t|_{\partial \mathcal{I}}=v_t|_{\partial \mathcal{I}}=0,$$
 we have
 \begin{align} \label{id_ux}
 \frac{\mathrm{d}}{\mathrm{d}t}\|(u_x,v_x)(t)\|^2
 +2\int_{\mathcal{I}}\frac{r^{2m}}{\tau}\left(\nu u_{xx}^2+\mu v_{xx}^2\right)=\sum_{q=8}^{12} \mathcal{K}_q,
 \end{align}
 where each term $\mathcal{K}_q$ in the decomposition will be defined and estimated below.
 First, we get
 \begin{align}
 \mathcal{K}_8:=2\int_{\mathcal{I}} u_{xx}r^m P_x
 \leq \epsilon  \|u_{xx}(t)\|^2
 +C(\epsilon) \|(\theta_x,\tau_x)(t)\|^2. \label{K8}
 \end{align}
 For the term
 \begin{equation*}
 \mathcal{K}_{9}:=
 2\int_{\mathcal{I}} \left[-u_{xx}\frac{v^2}{r}+v_{xx}\frac{uv}{r}\right],
 \end{equation*}
 we have from Sobolev's inequality and \eqref{E_1a} that
 \begin{align} \notag
 \mathcal{K}_{9}&\leq \epsilon \|(u_{xx},v_{xx})(t)\|^2
 +C(\epsilon)|(u,v)(t)\|_{L^{\infty}}^2\|v(t)\|^2\\ 
 &\leq \epsilon \|(u_{xx},v_{xx})(t)\|^2
 +C(\epsilon)\|(u_x,v_x)(t)\|\|v(t)\|^2.
 \label{K9}
 \end{align} 
 To estimate the term
 \begin{equation*}
 \mathcal{K}_{10}:=-2\int_{\mathcal{I}}
 u_{xx} r^m\left[\left(\frac{\nu (r^m u)_x}{\tau}\right)_x-
 \frac{\nu r^m u_{xx}}{\tau}\right],
 \end{equation*}
we first compute from \eqref{r_eq}, \eqref{E_theta} and \eqref{E_tau} that
\begin{align} \label{est_3.2}
\left|\left(\frac{\nu (r^m u)_x}{\tau}\right)_x-\frac{\nu r^m u_{xx}}{\tau}\right| 
\lesssim|(\tau_x,\theta_x)||(u_x,mu)|+|(u_x,mu\tau_x,mu)|.
\end{align}
It follows from \eqref{est_3.2}, \eqref{apriori5}, and \eqref{E_1a} that 
\begin{align} \notag
\mathcal{K}_{10}
&\leq \epsilon \|u_{xx}(t)\|^2+C(\epsilon)
\int_{\mathcal{I}}\left[\tau_x^2 u_x^2+(1+\theta_x^2)|(u_x,mu\tau_x,mu)|^2\right]\\ 
&\leq \epsilon \|u_{xx}(t)\|^2+C(\epsilon)
\left[1+\lrn \theta_x\rrn^2\right]\left[
\|(u_x,mu)(t)\|^2+\|\tau_x(t)\|^2\|(u_x,mu)(t)\|_{L^{\infty}}^2\right]\notag \\[1mm]
&\leq \epsilon \|u_{xx}(t)\|^2+C(\epsilon)
\left[\|(u_x,mu)(t)\|^2+\|(u_x,mu)(t)\|\|(u_{xx},u_x)(t)\|\right] \notag \\[1mm]
&\leq 2\epsilon \|u_{xx}(t)\|^2+C(\epsilon)\|(u_x,mu)(t)\|^2.
\label{K10}
\end{align}
  The term
  \begin{align*}
  \mathcal{K}_{11}
  := 4m\int_{\mathcal{I}}  r^{m-1}u \mu_x u_{xx}
  \end{align*}
  can be easily estimated as
  \begin{align}
  \mathcal{K}_{11}
  \leq \epsilon  \|u_{xx}(t)\|^2+C(\epsilon)\lrn\theta_x\rrn^2 \|mu(t)\|^2
  \leq \epsilon  \|u_{xx}\|^2+C(\epsilon)\|mu(t)\|^2. \label{K11}
  \end{align}
  The last term is 
  \begin{align*}
  \mathcal{K}_{12}:=
  -2\int_{\mathcal{I}}
   v_{xx} \left[r^m v_x \left(\frac{\mu r^m}{\tau}\right)_x+2\mu r^{m-1} v_x\
   -m(\mu r^{m-1} v)_x-\frac{\mu\tau v }{r^{2m}}\right].
  \end{align*}
  Similar to the derivation of \eqref{K10}, we can get
  \begin{align} \notag
  \mathcal{K}_{12}&\leq \epsilon  \|v_{xx}(t)\|^2
  +C(\epsilon)\int_{\mathcal{I}}\left[
  v_x^2(1+\theta_x^2+\tau_x^2)+v^2\right]\\
  &\leq 2\epsilon \|v_{xx}(t)\|^2+C(\epsilon)\|(v_x,v)(t)\|^2  .
  \label{K12}
  \end{align}
  Plugging \eqref{K8}--\eqref{K12} into 
  \eqref{id_ux} and taking $\epsilon$ sufficiently small yield
\begin{align}
\frac{\mathrm{d}}{\mathrm{d}t}\|(u_x,v_x)(t)\|^2+\left\|( u_{xx},v_{xx})(t)\right\|^2
\lesssim \|(\theta_x,\tau_x,u_x,mu,v_x,v)(t)\|^2+\|(u_x,v_x)(t)\|\|v(t)\|^2.
\label{est4}
\end{align}  
Integrating \eqref{est4} over $(0,t)$, we deduce from \eqref{E_1a} that
\begin{align*}
\|(u_x,v_x)(t)\|^2+\int_{0}^{t}\left\|( u_{xx},v_{xx})(s)\right\|^2\mathrm{d}s
\lesssim 1+\sup_{s\in[0,t]}\|(u_x,v_x)(s)\|,
\end{align*} 
from which we can conclude \eqref{E_uxvx}.
\end{proof}

The following lemma is the estimate on $\|\sqrt{c_v}\theta_x(t)\|$.
\begin{lemma}\label{lem_thex}
If the conditions listed in Lemma \ref{lem_bas} hold for a sufficiently small $\epsilon_1$, then 
\begin{equation} \label{E_thex}
\sup_{t\in[0,T]}\left\|\sqrt{c_v}\theta_x(t)\right\|^2 +\int_{0}^{T}
\|\theta_{xx}(t)\|^2\mathrm{d}t\leq C_5.
\end{equation}
\end{lemma}
\begin{proof}
 Multiply \eqref{cns.e} by $\theta_{xx}$, use the boundary condition for $\theta$, and integrate the resulting identity to get
 \begin{align}\label{id_thex}
 \frac{\mathrm{d}}{\mathrm{d}t}\|\sqrt{c_v}\theta_x(t)\|^2
 +2\int_{\mathcal{I}}\frac{\kappa r^{2m} \theta_{xx}^2}{\tau}
 =\sum_{q=13}^{15}\mathcal{K}_q,
 \end{align}
 where the term $\mathcal{K}_q$ will be given below.
 For the term
 \begin{align*}
 \mathcal{K}_{13}:=\int_{\mathcal{I}}
 2\theta_{xx}\left[\frac{\theta(r^m u)_x}{\tau}
 -\frac{\nu(r^m u)_x^2}{\tau}+2m\mu(r^{m-1} u^2)_x\right],
 \end{align*}
  it follows from \eqref{r_eq}, \eqref{E_theta}, \eqref{E_tau}, \eqref{E_1a} and \eqref{E_uxvx} that
\begin{align}\notag
\mathcal{K}_{13}&\leq 
\epsilon  \|\theta_{xx}(t)\|^2
+C(\epsilon)\int_{\mathcal{I}}|(u_x,mu,1)|^2|(u_x,mu)|^2\\ \notag
& \leq  \epsilon  \|\theta_{xx}(t)\|^2
+C(\epsilon)\|(u_x,mu,1)(t)\|_{L^{\infty}}^2\|(u_x,mu)\|^2\\[1mm]
& \leq  \epsilon \|\theta_{xx}(t)\|^2
+C(\epsilon)\|(u_x,mu,1)(t)\|^2\|(u_x,mu,u_{xx})(t)\|^2 \notag \\[1mm]
& \leq\epsilon \|\theta_{xx}(t)\|^2+C(\epsilon)\|(u_x,mu,u_{xx})(t)\|^2.
\label{K13}
\end{align}
Defining
\begin{align*}
\mathcal{K}_{14}:=\int_{\mathcal{I}}2\theta_{xx}\left[
-\frac{\mu r^{2m} w_x^2}{\tau} -\mu \tau\left(\frac{r^m v_x}{\tau}-\frac{v}{r^m}\right)^2\right],
\end{align*}
 we have from \eqref{E_1a}, \eqref{E_wx} and \eqref{E_uxvx} that
\begin{align}\notag
\mathcal{K}_{14}&\leq
\epsilon \|\theta_{xx}\|^2
+C(\epsilon)\int_{\mathcal{I}}|(w_x,v_x,v)|^4\\ \notag
& \leq  \epsilon \|\theta_{xx}\|^2
+C(\epsilon)\|(w_x,v_x,v)(t)\|_{L^{\infty}}^2
\|(w_x,v_x,v)(t)\|^2\\[1mm]
&\leq \epsilon \|\theta_{xx}\|^2+C(\epsilon)\|(w_{xx},v_{xx},w_x,v_x,v)(t)\|^2
\label{K14}
\end{align}
The term
\begin{align*}
\mathcal{K}_{15}:=\int_{\mathcal{I}}2 \theta_{xx}\left[\frac{\kappa r^{2m} \theta_{xx}}{\tau}-\left(\frac{\kappa r^{2m} \theta_x}{\tau}\right)_x\right]
\end{align*}
can be bounded by using \eqref{apriori5} as
\begin{align}\notag
\mathcal{K}_{15} &\leq 
\epsilon \|\theta_{xx}(t)\|^2
+C(\epsilon)\int_{\mathcal{I}} \theta_x^2\left[1+\theta_x^2+\tau_x^2\right]\\ \notag
&\leq 
\epsilon \|\theta_{xx}(t)\|^2+C(\epsilon)\left[1+\lrn \theta_x\rrn^2\right] \|(\theta_x,\tau_x)(t)\|^2\\[1mm]
&\leq 
\epsilon \|\theta_{xx}(t)\|^2+C(\epsilon)\|(\theta_x,\tau_x)(t)\|^2.
\label{K15}
\end{align}
We insert \eqref{K13}--\eqref{K15} into \eqref{id_thex} to obtain
\begin{align}
\frac{\mathrm{d}}{\mathrm{d}t}\|\sqrt{c_v}\theta_x(t)\|^2+\|\theta_{xx}(t)\|^2\lesssim 
\|(mu,v,\tau_x,u_x,v_x,w_x,\theta_x,u_{xx},v_{xx},w_{xx})(t)\|^2,
\label{est5}
\end{align}
which combined with \eqref{H1}, \eqref{E_1a}, \eqref{E_wx} and \eqref{E_uxvx} implies \eqref{E_thex}.
\end{proof}

Using the system \eqref{cns}, we can deduce the next lemma from Corollary \ref{cor_1}, Lemmas \ref{lem_wx}--\ref{lem_thex}. 
\begin{lemma}
If the conditions listed in Lemma \ref{lem_bas} hold for a sufficiently small $\epsilon_1$, then 
 \begin{equation} \label{E_1b}
 \sup_{t\in[0,T]}\left\|\tau_t(t)\right\|^2 +\int_{0}^{T}
 \|(\tau_{xt},
 u_{t},v_{t},w_{t},c_v\theta_t)(t)\|^2\mathrm{d}t\lesssim 1.
 \end{equation}
\end{lemma}
\begin{proof}
 We have from \eqref{cns.a}, \eqref{r_eq}, \eqref{E_theta}, and \eqref{E_tau} that
 \begin{align*}
 |\tau_t|\lesssim |(mu,u_x)|,\quad
 |\tau_{xt}|\lesssim |(u_{xx},u_x,mu,mu\tau_x)|,
 \end{align*}
 which combined with \eqref{E_1a} and \eqref{E_uxvx} yields 
 \begin{align*}
 \sup_{t\in[0,T]}\left\|\tau_t(t)\right\|^2 +\int_{0}^{T}
 \|\tau_{xt}(t)\|^2\mathrm{d}t\lesssim 1.
 \end{align*}
 By virtue of \eqref{r_eq}, \eqref{E_theta} and \eqref{E_tau}, we can deduce from \eqref{cns.e}  that
 \begin{align*}
 |c_v\theta_t|\lesssim 
 \left|\left(mu,u_x,\theta_{xx},\theta_x^2,\tau_x\theta_x,\theta_x,mu^2,u_x^2,w_x^2,v_x^2,v^2\right)\right|,
 \end{align*}
 which combined with \eqref{apriori5}, \eqref{E_1a}, \eqref{E_wx}, and \eqref{E_uxvx} implies
 \begin{align} \notag
 \|c_v\theta_t(t)\|^2&\lesssim
 \|(mu,u_x,\theta_{xx},\theta_x,\tau_x)(t)\|^2
 +\int_{\mathcal{I}}|(mu,u_x,w_x,v_x,v)|^4\\ \notag
 &\lesssim  \|(mu,u_x,\theta_{xx},\theta_x,\tau_x)(t)\|^2
 +\|(mu,u_x,w_x,v_x,v)\|^2\|(mu,u_x,w_x,v_x,v)\|_1^2\\[1mm]
 &\lesssim   \|(mu,u_x,\theta_{xx},\theta_x,\tau_x,u_{xx},w_x,w_{xx},v_{xx},v_x,v)(t)\|^2.
 \label{est2a}
 \end{align}
 Then it follows from Corollary \ref{cor_1}, Lemmas \ref{lem_wx}--\ref{lem_thex} that 
 \begin{align*}
 \int_{0}^{T}\|c_v\theta_t(t)\|^2\mathrm{d}t\lesssim 1.
\end{align*}

By virtue of \eqref{est2a} and \eqref{H1}, we have 
\begin{align}  \label{initial_1}
\|c_v\theta_t(0)\|\leq C,
\end{align}
for some $(\gamma-1)$-independent positive constant $C$.

 The other estimates in \eqref{E_1b} can be proved by a similar computation.
\end{proof}
\subsection{Estimates on second-order derivatives}
\label{subsec4}
In this subsection, we aim to derive the uniform bounds for $\|(\tau_{xx},u_{xx},v_{xx},w_{xx},\theta_{xx})(t)\|$.  
For this purpose, we make the estimates for $(u_t,v_t,w_t)$ in the next lemma. 
\begin{lemma} \label{lem_ut}
If the conditions listed in Lemma \ref{lem_bas} hold for a sufficiently small $\epsilon_1$, then
 \begin{equation} \label{E_utvtwt}
 \sup_{t\in[0,T]}\|(u_t,v_t,w_t)(t)\|^2 +
 \int_{0}^{T}\|(u_{xt},\tau_{tt},v_{xt},w_{xt})(t)\|^2\mathrm{d}t\lesssim 1.
 \end{equation}
\end{lemma}
\begin{proof} The proof is divided into three steps.
	
\noindent {\em Step 1.} 
 We let $\partial_t$ act on \eqref{cns.b}, and multiply the resulting identity by $u_t$ to yield
 \begin{align} \label{id4.1} 
 \left[\frac{1}{2}u_t^2\right]_t+\frac{\nu r^{2m}u_{xt}^2}{\tau}
 +\left[u_tr^m P_t-u_t r^m \left(\frac{\nu (r^m u)_x}{\tau}\right)_t\right]_x
 =J_1+J_2+J_3,
 \end{align}
 with 
 \begin{align*}
  J_1&:=
  -(r^m)_xu_t\left[\frac{\nu(r^m u)_x}{\tau}\right]_t,\\[0.5mm]
 J_2&:=
 r^m u_{xt}\left[P_t-\left(\frac{\nu}{\tau}\right)_t(r^m u)_x
 -\frac{\nu}{\tau}\left((r^m u)_{xt}-r^m u_{xt}\right)\right],\\[0.5mm]
  J_3&:=u_t\left[\left(\frac{v^2}{r}\right)_t-(r^m)_tP_x+(r^m)_xP_t+
  (r^m)_t\left(\frac{\nu (r^mu)_x}{\tau}\right)_x-2m\left(r^{m-1}u \mu_x\right)_t\right].
 \end{align*}
 It follows  from \eqref{E_1a}, \eqref{E_wx} and \eqref{E_uxvx} that
 \begin{align} \label{E_uvw}
 \lrn (u,v,w)\rrn\lesssim 1.
 \end{align}
  To estimate $J_1$ and $J_2$, we first get from \eqref{r_eq} that
  \begin{align*}
  \left|(r^m u)_{xt}-r^m u_{xt}\right|
  &\lesssim\left|((r^m)_{tx}u,(r^m)_xu_t,(r^m)_tu_x)\right|
  \lesssim |(mu^2,u_t,uu_x)|,\\
  \left|(r^m u)_{xt}\right|
  &\lesssim|(mu^2,u_t,uu_x,u_{xt})|.
  \end{align*}
  Then we have from \eqref{apriori5} and \eqref{E_uvw} that
  \begin{align} 
  J_1\lesssim 
  |u_{t}|
  \left|\left(\theta_t\tau_t,\tau_t^2,mu,u_t,u_x,u_{xt}\right)\right|
  \lesssim  \epsilon u_{xt}^2+C(\epsilon)
  \left|\left(\tau_t,\tau_t^2,mu,u_t,u_x\right)\right|^2,
  \label{J1}
  \end{align}
 and
  \begin{align} 
  J_2\lesssim 
  |u_{xt}|
  \left|\left(\theta_t,\tau_t,\theta_t\tau_t,\tau_t^2,mu^2,u_t,uu_x\right)\right|
  \lesssim  \epsilon u_{xt}^2+C(\epsilon)
  \left|\left(\theta_t,\tau_t,\tau_t^2,mu,u_t,u_x\right)\right|^2.
  \label{J2}
  \end{align} 
 According to \eqref{cns.b}, we have from \eqref{E_theta},  \eqref{E_tau}, \eqref{apriori5}, and \eqref{E_uvw} that
 \begin{align*}
 \left|\left[\frac{\nu (r^m u)_x}{\tau}\right]_x\right|
 \lesssim \left|\left(u_t,v^2,\theta_t,\tau_t,\theta_x\right)\right|, 
 \end{align*}
and 
 \begin{align} 
 J_3
 \lesssim |u_t|\left|\left(v_t,v,\theta_x,\tau_x,\theta_t,\tau_t,u_t,\theta_{xt}\right)\right|\lesssim 
 \left|(u_t,v_t,v,\theta_x,\tau_x,\theta_t,\tau_t,\theta_{xt})\right|^2.
 \label{J3}
 \end{align}
 Integrating \eqref{id4.1} over $(0,t)\times\mathcal{I}$, using \eqref{J1}--\eqref{J3}, \eqref{E_1a}--\eqref{E_uxvx},  \eqref{apriori5}, and \eqref{E_1b} yield
 \begin{align}
 \|u_t(t)\|^2+\int_{0}^{t}\|u_{xt}\|^2
 \lesssim 1+\int_{0}^{t}\|\theta_{xt}\|^2+ \int_{0}^{t}\|\tau_t\|_{L^{\infty}}^2\|\tau_t\|^2 \lesssim 1+\int_{0}^{t}\|\tau_t\|_1^2\lesssim 1.  \label{E_ut}
 \end{align}

By virtue of \eqref{cns.a}, $\tau_{tt}=(r^m u)_{xt}$, and hence 
\begin{align*}
|\tau_{tt}|\lesssim |u_{xt}|+|(u_t,uu_x)|+|mu||(u_x,1)|.
\end{align*} 
Combining this with \eqref{E_1a},\eqref{E_1b}, \eqref{E_ut},  and \eqref{E_uvw} implies
\begin{align}
\int_{0}^{T}\|\tau_{tt}(t)\|^2\mathrm{d}t\lesssim 1. \label{E_itautt}
\end{align} 
 
\noindent {\em Step 2.} 
 Letting $\partial_t$ act on \eqref{cns.c}, and  multiplying the resulting identity by $v_t$  yield
 \begin{align}  \label{id4.2}
 \left[\frac{1}{2}v_t^2\right]_t+\frac{\mu r^{2m}v_{xt}^2}{\tau} -\left[v_t r^m \left(\frac{\mu r^m v_x}{\tau}\right)_t\right]_x =J_4+J_5,
 \end{align}
 with 
 \begin{align*}
 J_4&:=v_t\left[-\left(\frac{uv}{r}\right)_t+(r^m)_t\left(\frac{\mu r^m v_x}{\tau}\right)_x -\left(\frac{\mu \tau v}{r^{2m}}\right)_t\right],\\
 J_5&:= v_t \left[-(r^m)_x\left(\frac{\mu r^m v_x}{\tau}\right)_t+2\left(\mu r^{m-1}v_x\right)_t-m\left(\mu r^{m-1}v\right)_{xt}\right]-r^m v_{xt}v_x\left[\frac{\mu r^m }{\tau}\right]_t.
 \end{align*}
 By virtue of \eqref{E_theta}, \eqref{E_tau} and \eqref{E_uvw},  we derive from \eqref{cns.c} that
 \begin{align*}
 \left|\left[\frac{\mu r^m v_x}{\tau}\right]_x\right|
 \lesssim \left|\left(v_t,v,\theta_x,v_x\right)\right|, 
 \end{align*} 
 and
 \begin{align}
 J_4\lesssim |v_t|\left|\left(u_t,v_t,v,v_x,\theta_x,\tau_t,\theta_t\right)\right|.
 \label{J4}
 \end{align}
It follows from \eqref{apriori5} and \eqref{E_uvw} that
 \begin{align}\notag
 J_5&\lesssim
 |v_t|\left|\left(v_{xt},
 \tau_t v_x,v_x,v_t,\theta_{xt},\theta_x,\theta_t,u_x,mu\right)\right|
 +|v_{xt}v_x||(\theta_t,\tau_t,u)|\\
 &\lesssim \epsilon v_{xt}^2+C(\epsilon)
 \left|\left(v_t, \tau_t v_x,v_x,\theta_{xt},\theta_x,\theta_t,u_x,mu\right)\right|^2.
 \label{J5}
 \end{align} 
 Integrating \eqref{id4.2} over $(0,t)\times\mathcal{I}$, using \eqref{J4}--\eqref{J5}, \eqref{E_1a}--\eqref{E_uxvx}, \eqref{apriori5}, and \eqref{E_1b} yield
 \begin{align}
 \|v_t(t)\|^2+\int_{0}^{t}\|v_{xt}\|^2
 \lesssim 1+\int_{0}^{t}\|\theta_{xt}\|^2+
 \int_{0}^{t}\|\tau_t\|_{L^{\infty}}^2\|v_x\|^2
 \lesssim 1+\int_{0}^{t}\|\tau_t\|_1^2\lesssim 1.
 \label{E_vt}
 \end{align}

\noindent {\em Step 3.} 
 Let $\partial_t$ act on \eqref{cns.d}, and multiply the resulting identity by $w_t$ to get
 \begin{align}\label{id4.3}
 \left[\frac{1}{2}w_t^2\right]_t+\frac{\mu r^{2m}w_{xt}^2}{\tau}
 -\left[w_t r^m \left(\frac{\mu r^m w_x}{\tau}\right)_t\right]_x
 =J_6+J_7,
 \end{align}
    with
    \begin{align*}
    J_6&:=w_t\left(r^m\right)_t\left[\frac{\mu r^m w_x}{\tau}\right]_x,\\
    J_7&: =-r^m w_x w_{xt}\left[\frac{\mu r^m}{\tau}\right]_t
    - w_t\left[
    \left(r^m\right)_x\left(\frac{\mu r^m w_x}{\tau}\right)_t
    -m\left(\mu r^{m-1}w_x\right)_t\right].
    \end{align*}
According to \eqref{cns.d} and \eqref{E_uvw}, we infer
\begin{align}
\label{J6}
J_6&\lesssim |w_t||(w_t,w_x)|\lesssim |(w_t,w_x)|^2,\\[0.5mm] \notag
J_7&\lesssim 
|w_{xt}||(\theta_tw_x,\tau_tw_x,w_x,w_t)|
+|w_t||(\theta_tw_x,\tau_tw_x,w_x)|\\
&\lesssim \epsilon w_{xt}^2+C(\epsilon)
|(w_t,\theta_tw_x,w_x,\tau_t w_x)|^2.
\label{J7}
\end{align}
 Integrating \eqref{id4.3} over $(0,t)\times\mathcal{I}$,
 using \eqref{J4}--\eqref{J5},
 \eqref{E_1a}--\eqref{E_uxvx}, \eqref{apriori5}, and \eqref{E_1b} yield
 \begin{align}
 \|w_t(t)\|^2+\int_{0}^{t}\|w_{xt}\|^2
 \lesssim 1+
 \int_{0}^{t}\|\tau_x\|_{L^{\infty}}^2\|w_x\|^2
 \lesssim 1+\int_{0}^{t}\|\tau_t\|_1^2\lesssim 1.
 \label{E_wt}
 \end{align}
Combine the estimates \eqref{E_ut}, \eqref{E_itautt}, \eqref{E_vt} and \eqref{E_wt} to
derive  \eqref{E_utvtwt}.
\end{proof}

Noting that the equations for $(u,v,w)$ are parabolic, we deduce the uniform  bounds for the $L_x^2$-norms of $(u_{xx},v_{xx},w_{xx})$ in the following lemma.
\begin{lemma}   If the conditions listed in Lemma \ref{lem_bas} hold for a sufficiently small $\epsilon_1$, then
 \begin{equation} \label{E_uxx}
 \sup_{t\in[0,T]}\|(u_{xx},v_{xx},w_{xx},\tau_{xt})(t)\|\leq C_6.
 \end{equation}
\end{lemma}
\begin{proof}
It follows from \eqref{r_eq}, \eqref{E_theta}, and \eqref{E_tau} that
\begin{align}
|(r^m)_x|\lesssim m,\quad
|(r^m)_{xx}|\lesssim |(\tau_x,m)|,\quad
|(r^m)_{xxx}|\lesssim |(\tau_{xx},\tau_x,m)|,
\end{align}
and hence
\begin{align} \label{est4.1}
|(r^m u)_{xx}-r^m u_{xx}|\lesssim |(\tau_x u,mu,u_x)|.
\end{align}
According to \eqref{cns.b}, we get
\begin{align*}
\frac{\nu r^{2m} u_{xx}}{\tau}
= u_t-\frac{v^2}{r}+r^mP_x+2mr^{m-1}u\mu_x
-r^m\left[\left(\frac{\nu(r^m u)_x}{\tau}\right)_x-
\frac{\nu r^m u_{xx}}{\tau}\right].
\end{align*}
Using \eqref{E_theta},  \eqref{E_tau} and \eqref{est4.1}, we obtain
\begin{align*}
|u_{xx}|\lesssim
\left|\left(u_t,v,\theta_x,\tau_x\right)\right|+
\left|\left(\theta_x\tau_t,\tau_x\tau_t\right)\right|+
\left|\left(\tau_x u,mu,u_x\right)\right|,
\end{align*}
which combined
with Corollary \ref{cor_1}, Lemmas \ref{lem_wx}--\ref{lem_ut}
yields
\begin{align*}
 \sup_{t\in[0,T]}\|u_{xx}(t)\|\lesssim 1.
\end{align*}
Since the proof of the other estimates in \eqref{E_uxx} can be shown in a similar way, we omit it.
\end{proof}

The following lemma concerns the bounds for  the $L^2_x$-norms of $c_v\theta_t$ and $\theta_{xx}$.
\begin{lemma}\label{lem_thet}
   If the conditions listed in Lemma \ref{lem_bas} hold for a sufficiently small $\epsilon_1$, then
 \begin{equation} \label{E_thet}
 \sup_{t\in[0,T]}\left\|(c_v\theta_t,\theta_{xx})(t)\right\|^2 +\int_{0}^{T}
 \|\sqrt{c_v}\theta_{xt}(t)\|^2\mathrm{d}t\leq C_7.
 \end{equation}
\end{lemma}
\begin{proof}
Let $\partial_t$ act on \eqref{cns.e} and multiply the resulting identity with $c_v\theta_t$ to discover
 \begin{align}\notag
 &\left[\frac{\left(c_v\theta_t\right)^2}{2}\right]_t+\frac{c_v \kappa r^{2m}\theta_{xt}^2}{\tau}
 -\left[c_v\theta_t\left(\frac{\kappa r^{2m} \theta_x}{\tau}\right)_t\right]_x\\
 &\quad=c_v\theta_t\mathcal{Q}_t-c_v\theta_t\left(P(r^m u)_x\right)_t-
 c_v\theta_x\theta_{xt}\left[\frac{\kappa r^{2m}}{\tau}\right]_t. 
 \notag
 \end{align}
By virtue of the boundary conditions \eqref{bdy}, \eqref{E_1b}, \eqref{E_thex}, \eqref{E_1a}, and \eqref{initial_1}, we integrate the last identity to get
\begin{align}\notag
&\|c_v\theta_t(t)\|^2+\int_{0}^{t}\|\sqrt{c_v}\theta_{xt}\|^2\\ \notag
&\quad\lesssim  \|c_v\theta_t(0)\|^2+\int_{0}^{t}\|\left(c_v\theta_t,\mathcal{Q}_t,
\left(P(r^m u)_x\right)_t\right)\|^2
+\int_{0}^{t}\int_{\mathcal{I}}c_v\theta_x^2|(\theta_t,\tau_t,mu)|^2\\ \notag
&\quad \lesssim 1+
\int_{0}^{t}\|\left(\mathcal{Q}_t,\left(P(r^m u)_x\right)_t\right)\|^2
+\sup_{s\in[0,T]}\|\sqrt{c_v}\theta_x(s)\|^2\int_{0}^{t}\|(\theta_t,\tau_t,mu)\|_1^2\\
&\quad \lesssim 1+\int_{0}^{t}\|\left(\mathcal{Q}_t,
\theta_t\tau_t,\tau_t^2,\tau_{tt}\right)\|^2. \label{est4.2}
\end{align}
Using \eqref{E_theta},   \eqref{E_tau} and \eqref{E_uvw},
after some elementary calculations, we have
\begin{align}\notag
|\mathcal{Q}_t|\lesssim~&
|(\theta_t,\tau_t)|\tau_t^2+|\tau_t\tau_{tt}|+
|\theta_t||(u_x,mu)|+|(u_x,mu,u_t,u_{xt},u_xu_t)|+|w_xw_{xt}|\\
&+w_x^2|(\theta_t,\tau_t,mu)|
+|(\theta_t,\tau_t)||(v_x,v)|^2
+|(v_x,v)||(v_{xt},v_t,v_x,v_x\tau_t,v)|, \notag
\end{align}
and
\begin{align*}
\mathcal{Q}_t^2\lesssim~&
|(\theta_t,\tau_t,mu,w_x,v_x,v)|^6+\tau_t^2\tau_{tt}^2+w_x^2w_{xt}^2
+v_x^2v_{xt}^2\\
&+
|(\theta_t,u_x,\tau_t,u_t,v_x,v_t)|^4+
|(\theta_t,u_x,mu,u_t,u_{xt},v_{xt},v_t,v_x,v)|^2.
\end{align*}
Employing Sobolev's inequality, \eqref{apriori5}, Corollary \ref{cor_1}, and Lemmas \ref{lem_wx}--\ref{lem_ut} yields
\begin{align}\notag 
&\int_{0}^{t}\int_{\mathcal{I}} |(\theta_t,\tau_t,mu,w_x,v_x,v)|^6\\  \label{est4.3}
&\quad\lesssim \sup_{[0,t]}\|(\theta_t,\tau_t,mu,w_x,v_x,v)\|^4
\int_{0}^{t}\|(\theta_t,\tau_t,mu,w_x,v_x,v)\|_1^2
\lesssim 1,\\ \notag
&\int_{0}^{t}\int_{\mathcal{I}}|(\theta_t,\theta_x,u_x,\tau_t,u_t,v_x,v_t)|^4\\ 
&\quad\lesssim \sup_{[0,t]}\|(\theta_t,\theta_x,u_x,\tau_t,u_t,v_x,v_t)\|^2
\int_{0}^{t}\|(\theta_t,\theta_x,u_x,\tau_t,u_t,v_x,v_t)\|_1^2
\lesssim 1, \label{est4.4}
\end{align}
and 
\begin{align}
\int_{0}^{t}\int_{\mathcal{I}}\left[\tau_t^2\tau_{tt}^2+w_x^2w_{xt}^2
+v_x^2v_{xt}^2\right]\lesssim 
\sup_{[0,t]}\|(\tau_t,w_x,v_x)\|_1^2\int_{0}^{t}\|(\tau_{tt},w_{xt},v_{xt})\|^2
\lesssim 1. \label{est4.5}
\end{align}
Combining \eqref{est4.2}--\eqref{est4.5} implies
 \begin{equation} \label{E_thet1}
 \sup_{t\in[0,T]}\left\|c_v\theta_t(t)\right\|^2 +\int_{0}^{T}
 \|\sqrt{c_v}\theta_{xt}(t)\|^2\mathrm{d}t\lesssim 1.
 \end{equation}

By virtue of the equation \eqref{cns.e}, we have
\begin{align*}
\frac{\kappa r^{2m} \theta_{xx}}{\tau}
=c_v\theta_t+P\tau_t
-\left(\frac{\kappa r^{2m}}{\tau}\right)_x \theta_x-\mathcal{Q},
\end{align*}
which implies
\begin{align*}
|\theta_{xx}|\lesssim |(c_v\theta_t,\tau_t)|+|\theta_x||(\theta_x,\tau_x,m)|
+|(\tau_t^2,u_x,mu,w_x^2,v_x^2,v^2)|.
\end{align*}
We use this estimate and \eqref{est4.4}, \eqref{E_thet1} to have
 \begin{equation} \label{E_thet2}
 \sup_{t\in[0,T]}\left\|\theta_{xx}(t)\right\|^2 \lesssim 1.
 \end{equation}
The estimate \eqref{E_thet} then follows from \eqref{E_thet1} and \eqref{E_thet2}.
\end{proof}

In the following lemma we make an estimate for $((u_{xxx},v_{xxx},w_{xxx},\theta_{xxx})$ in $L^{2}(0,T;L^2(\mathcal{I}))$.
\begin{lemma} \label{lem_iuxxx}
 If the conditions listed in Lemma \ref{lem_bas} hold for a sufficiently small $\epsilon_1$, then
 \begin{equation} \label{E_iuxxx}
 \int_{0}^{T}\left\|\left(u_{xxx},v_{xxx},w_{xxx},
 \frac{\theta_{xxx}}{\sqrt{c_v}}\right)(t)\right\|^2
 \mathrm{d}t\lesssim 1+\int_{0}^{T}\|\tau_{xx}(t)\|^2\mathrm{d}t.
 \end{equation}
\end{lemma}
\begin{proof}
Differentiating \eqref{cns.b} with respect to $x$, we have
 \begin{align*}\notag
 &\frac{r^{2m}\nu u_{xxx}}{\tau}+(r^m)_x\left[\frac{\nu\tau_t}{\tau}\right]_x+r^m\left[\left(\frac{\nu}{\tau}\right)_x\tau_t\right]_x+r^m\left(\frac{\nu}{\tau}\right)_x(r^m u)_{xx}\\
&\quad  =
 u_{xt}-\left[\frac{v^2}{r}\right]_x+\left(r^m P_x\right)_x+2m\left(r^{m-1}u \mu_x\right)_x-\frac{r^m \nu }{\tau}\left[(r^m u)_{xxx}-r^m u_{xxx}\right].
 \end{align*}
 By a direct computation, applying \eqref{E_theta}, \eqref{E_tau}, \eqref{apriori5}, and \eqref{E_uvw} gives
\begin{align}
\notag |u_{xxx}|
\lesssim ~&|u_{xt}|+|(v_x,v,\theta_x,\tau_x,\theta_{xx},\tau_{xx},\tau_x^2,u_x)|
+|u||(\tau_{xx},\tau_x,m)|+|u_x||(\tau_x,1)|+|u_{xx}|\\
&+ |(\tau_{tx},\theta_x\tau_t,\tau_x\tau_t)|
+|(\theta_{xx}\tau_t,\tau_{xx}\tau_t,\tau_t,\tau_x^2\tau_t,\theta_x\tau_{xt},\tau_x\tau_{xt})|
+|\tau_{tx}||(\theta_x,\tau_x)|. \label{est_uxxx}
\end{align} 
Applying \eqref{E_uvw}, Corollary \ref{cor_1} and Lemmas \ref{lem_wx}--\ref{lem_thet}, we infer
\begin{align*}
\int_{0}^{T}\|u_{xxx}\|^2
\lesssim 1+\int_{0}^{T}\|\tau_{xx}\|^2
+\int_{0}^{T}\int_{\mathcal{I}} \left[\tau_x^4+\theta_{xx}^2\tau_t^2+
\tau_{xx}^2\tau_t^2+\tau_x^4\tau_t^2+\tau_x^2\tau_{xt}^2\right].
\end{align*}
The last term on the right can be estimated by employing the Sobolev's inequality as
\begin{align*} \notag
&\int_{0}^{T}\int_{\mathcal{I}} \left[\theta_{xx}^2\tau_t^2+
\tau_{xx}^2\tau_t^2+\tau_x^4\tau_t^2+\tau_x^4+\tau_x^2\tau_{xt}^2\right]\\
&\quad \lesssim
\sup_{t\in[0,T]}\|\tau_t\|_{L^{\infty}}^2\int_{0}^{T}\left[\|(\theta_{xx},\tau_{xx})\|^2+\|\tau_{xx}\|\|\tau_x\|^3\right]
+\int_{0}^{T}\|\tau_x\|\|\tau_{xx}\|\|(\tau_x,\tau_{tx})\|^2\\
&\quad \lesssim
1+\int_{0}^{T}\|\tau_{xx}\|^2.
\end{align*}
Here we have used \eqref{E_1a}, \eqref{E_thex} and \eqref{E_uxx}. Hence we get
\begin{align}
\int_{0}^{T}\|u_{xxx}\|^2
\lesssim 1+\int_{0}^{T}\|\tau_{xx}\|^2.
\end{align}

The estimate for $(v,w)$ in \eqref{E_iuxxx} can be obtained similarly.

We next show the estimate of $\theta_{xxx}$ in \eqref{E_iuxxx}.
Differentiate \eqref{cns.e} with respect to $x$ to get
\begin{align*}
\frac{\kappa r^{2m}\theta_{xxx}}{\tau}
=c_v\theta_{xt}+\left(P\tau_t\right)_x-
\left[\frac{\kappa r^{2m}}{\tau}\right]_{xx}\theta_x
-2\left[\frac{\kappa r^{2m}}{\tau}\right]_{x}\theta_{xx}-\mathcal{Q}_x,
\end{align*}
which combined with \eqref{apriori5} yields
\begin{align}\label{est_thexxx}
\frac{|\theta_{xxx}|}{\sqrt{c_v}}&\lesssim 
\left|\sqrt{c_v}\theta_{xt}\right|
+\left|\left(\theta_x\tau_t,\tau_x\tau_t,\tau_{xt}\right)\right|
+\left|\left(\theta_{xx},\tau_{xx},\tau_x^2,\theta_x,\tau_x\right)\right|
+\left|\left(\theta_{xx},\tau_{x}\theta_{xx},\mathcal{Q}_x\right)\right|,\\ \label{est_thext1}
\left|\sqrt{c_v}\theta_{xt}\right|&\lesssim 
\frac{|\theta_{xxx}|}{\sqrt{c_v}}
+\left|\left(\theta_x\tau_t,\tau_x\tau_t,\tau_{xt}\right)\right|
+\left|\left(\theta_{xx},\tau_{xx},\tau_x^2,\theta_x,\tau_x\right)\right|
+\left|\left(\theta_{xx},\tau_{x}\theta_{xx},\mathcal{Q}_x\right)\right|.
\end{align}
Applying \eqref{E_uvw}, Corollary \ref{cor_1} and Lemmas \ref{lem_wx}--\ref{lem_thet}, we infer
\begin{align*}
\int_{0}^{T}\left\|\frac{\theta_{xxx}}{\sqrt{c_v}}\right\|^2
&\lesssim 1+\int_{0}^{T}\|\tau_{xx}\|^2
+\int_{0}^{T}\int_{\mathcal{I}} \left[\theta_{xx}^2\tau_x^2+\mathcal{Q}_x^2\right]\\
&\lesssim 1+\int_{0}^{T}\|\tau_{xx}\|^2
+\sup_{[0,T]}\|\theta_{xx}\|^2\int_{0}^{T}\|\tau_{xx}\|\|\tau_x\|
+\int_{0}^{T}\|\mathcal{Q}_x\|^2\\
&\lesssim 1+\int_{0}^{T}\|\tau_{xx}\|^2
+\int_{0}^{T}\|\mathcal{Q}_x\|^2.
\end{align*}
To conclude \eqref{E_iuxxx}, it remains to prove
\begin{align}
\int_{0}^{T}\|\mathcal{Q}_x\|^2\lesssim 1. \label{est4.6}
\end{align}
According to the definition of $\mathcal{Q}$, we can compute that
\begin{align*} 
|\mathcal{Q}_x|\lesssim 
~& \tau_t^2|(\theta_x,\tau_x)|+|\tau_t\tau_{xt}|+
|\theta_x||(mu,u_x)|+|(mu,\tau_x,u_x,u_{xx},u_x^2)|\\  
&+w_x^2|(\theta_x,\tau_x,1)|+|w_xw_{xx}|+
|(\theta_x,\tau_x)||(v_x,v)|^2+|(v_x,v)||(v_{xx},v_x)|,
\end{align*}
which yields
\begin{align} \notag
\mathcal{Q}_x^2\lesssim~&
|(\theta_x,\tau_x,\tau_t,w_x,v,v_x)|^6
+|(\theta_x,mu,u_x,w_x,v,v_x)|^4\\
&+|(mu,\tau_x,u_x)|^2+\tau_t^2\tau_{xt}^2+w_x^2w_{xx}^2+v_x^2v_{xx}^2+u_{xx}^2+v_{xx}^2.\label{est_Qx2}
\end{align}
Applying \eqref{est4.3}--\eqref{est4.5}, 
Corollary \ref{cor_1}, and Lemmas \ref{lem_wx}--\ref{lem_thet}, we
can deduce \eqref{est4.6} and therefore complete the proof of this lemma.
\end{proof}

\begin{lemma} \label{lem_tauxx}
If the conditions listed in Lemma \ref{lem_bas} hold for a sufficiently small $\epsilon_1$, then
 \begin{align} \label{E_tauxx}
 \sup_{t\in[0,T]}\left\|\tau_{xx}(t)\right\|^2 +
 \int_{0}^{T}
 \left\|\left(\tau_{xx},u_{xxx},\tau_{xxt},v_{xxx},w_{xxx},
 \frac{\theta_{xxx}}{\sqrt{c_v}}\right)(t)\right\|^2
\mathrm{d}t\leq C_8.
 \end{align}
\end{lemma}
\begin{proof}
 Let $\partial_x$ act on \eqref{id2.2} and multiply the resulting identity
 by $(\nu \tau_x/\tau)_x$ to find
\begin{align*}
&\left[\frac{1}{2}\left(\frac{\nu \tau_x}{\tau}\right)_x^2\right]_t
+\frac{\theta}{\nu \tau}\left[\frac{\nu \tau_x}{\tau}\right]_x^2
-\left[\frac{\nu \tau_x}{\tau}\right]_x
\left[\frac{\nu'(\theta)}{\tau}\left(\tau_x\theta_t-\tau_t\theta_x\right)\right]_x\\
&=-\left[\frac{\nu \tau_x}{\tau}\right]_x
\left[-\frac{\theta_{xx}}{\tau}-\frac{2\theta \tau_x^2}{\tau^3}
+\frac{2\theta_x\tau_x}{\tau^2}-\frac{\theta \tau_x}{\nu \tau}\left(\frac{\nu}{\tau}\right)_x\right]\\
&\quad +\left[\frac{\nu \tau_x}{\tau}\right]_x\left[
\frac{u_t}{r^{m}}-\frac{v^2}{r^{m+1}}
+\frac{2m u\mu_x}{r}\right]_x.
\end{align*}
Integrating the last identity and using Cauchy's inequality yield
\begin{align}\notag
&\left\|\left(\frac{\nu \tau_x}{\tau}\right)_x(t)\right\|^2
+\int_{0}^{t}\left\|\left(\frac{\nu \tau_x}{\tau}\right)_x\right\|^2\\ \notag
&\quad\lesssim \left\|\left.\left(\frac{\nu \tau_x}{\tau}\right)_x\right|_{t=0}\right\|^2+\int_{0}^{t}\int_{\mathcal{I}}
\left[|(\theta_{xx},\tau_x^2,\theta_x\tau_x)|^2+
|(u_{xt},u_t,v_x,v,u_x\theta_x,\theta_{xx},\theta_x)|^2\right]\\
&\quad\quad\ +
\int_{0}^{t}\int_{\mathcal{I}}\left[\left|(\tau_x\theta_{tx},\tau_t\theta_{xx})\right|^2+
|(\tau_{xx}\theta_t,\tau_{tx}\theta_x)|^2+
\left|(\tau_x\theta_t,\tau_t\theta_x)\right|^2
\left|(\tau_x,\theta_x)\right|^2\right]. \label{est4.7}
\end{align}
We only show the estimates for the last three terms on the right-hand side of \eqref{est4.7}, since the other terms can be easily treated by using Corollary \ref{cor_1} and Lemmas \ref{lem_wx}--\ref{lem_thet}.
The term
\begin{align*}
\int_{0}^{t}\int_{\mathcal{I}}
\left|(\tau_x\theta_t,\tau_t\theta_x)\right|^2
\left|(\tau_x,\theta_x)\right|^2
\end{align*}
can be estimated by using \eqref{apriori5} and \eqref{est4.4} as
\begin{align} \label{est4.8}
\int_{0}^{t}\int_{\mathcal{I}}
\left|(\tau_x\theta_t,\tau_t\theta_x)\right|^2
\left|(\tau_x,\theta_x)\right|^2
\lesssim \int_{0}^{t}\int_{\mathcal{I}}
|(\tau_x,\tau_t,\theta_x)|^4
\lesssim \epsilon\int_{0}^{t}\|\tau_{xx}\|^2+C(\epsilon).
\end{align}
It follows from \eqref{E_1a}, \eqref{E_thex}, \eqref{E_1b}, and \eqref{E_thet} that
\begin{align}\notag
&\int_{0}^{t}\int_{\mathcal{I}}
\left|(\tau_x\theta_{tx},\tau_t\theta_{xx})\right|^2
\lesssim \int_{0}^{t}\|(\tau_x,\tau_t)\|\|(\tau_{xx},\tau_{xt})\|
\|(\theta_{xt},\theta_{xx})\|^2\\  
&\quad \lesssim \epsilon\int_{0}^{t}\|(\tau_{xx},\tau_{xt})\|^2+C(\epsilon)
\int_{0}^{t}\|(\theta_{xt},\theta_{xx})\|^2
\lesssim \epsilon\int_{0}^{t}\|\tau_{xx}\|^2+C(\epsilon).\label{est4.9}
\end{align}
In view of \eqref{apriori1} and \eqref{apriori5}, we have
\begin{align}
\int_{0}^{t}\int_{\mathcal{I}}\left|(\tau_{xx}\theta_t,\tau_{tx}\theta_x)\right|^2
&\lesssim \lrn (\theta_t,\theta_x)\rrn^2 \int_{0}^{t}
\|(\tau_{xx},\tau_{xt})\|^2\lesssim 1+
(\gamma-1)^{\frac{1}{2}}  N^2\int_{0}^{t}\|\tau_{xx}\|^2. 
\label{est4.10}
\end{align}
Thanks to \eqref{est4.7}--\eqref{est4.10} and \eqref{apriori2}, we derive
\begin{align*}
\left\|\left(\frac{\nu \tau_x}{\tau}\right)_x(t)\right\|^2
+\int_{0}^{t}\left\|\left(\frac{\nu \tau_x}{\tau}\right)_x\right\|^2
\lesssim C(\epsilon)
+\left(\epsilon+ \epsilon_1^2\right)\int_{0}^{t}\|\tau_{xx}\|^2.
\end{align*}
Noting that
\begin{align*}
|\tau_{xx}|
\lesssim \left|\left(\frac{\nu \tau_x}{\tau}\right)_x\right|
+\left|\left(\theta_x\tau_x,\tau_x^2\right)\right|,
\end{align*}
 taking $\epsilon$ and $\epsilon_1$ sufficiently small, we get
\begin{align}\notag
\left\|\tau_{xx}(t)\right\|^2+\int_{0}^{t}\left\|\tau_{xx}\right\|^2
&\lesssim 1+\int_{\mathcal{I}}\tau_x^2(\theta_x^2+\tau_x^2)+
\int_{0}^{t}\int_{\mathcal{I}}\tau_x^2(\theta_x^2+\tau_x^2)\\
&\lesssim 1+\sup_{s\in[0,t]}\|\tau_x(s)\|^2_{L^{\infty}}
\lesssim 1+\sup_{s\in[0,t]}\|\tau_{xx}(s)\|. \label{est4.11}
\end{align}
Apply Cauchy's inequality to \eqref{est4.11} to get
\begin{align*}
\sup_{t\in[0,T]}\left\|\tau_{xx}(t)\right\|^2
+\int_{0}^{T}\left\|\tau_{xx}\right\|^2\lesssim 1,
\end{align*}
which combined with \eqref{E_iuxxx} gives \eqref{E_tauxx}.
\end{proof}

\subsection{Estimates on the third-order derivatives}
\label{subsec5}
This subsection is devoted to deriving the estimates for $(\tau_{xxx},u_{xxx},v_{xxx},w_{xxx},\theta_{xxx})$.  
To do this, we first make the estimates for $(u_{xt},v_{xt},w_{xt})$ in the following lemma. 
\begin{lemma} \label{lem_uxt}
If the conditions listed in Lemma \ref{lem_bas} hold for a sufficiently small $\epsilon_1$, then
 \begin{align}\label{E_uxt}
 &\sup_{t\in[0,T]}\|(u_{xt},v_{xt},w_{xt})(t)\|^2 +\int_{0}^{T}
 \|(u_{xxt},v_{xxt},w_{xxt})(t)\|^2\mathrm{d}t\lesssim 1,\\ \label{E_uxxx}
 &\sup_{t\in[0,T]}\|(u_{xxx},\tau_{tt},\tau_{xxt},v_{xxx},w_{xxx})(t)\|\leq C_9.
 \end{align}
\end{lemma}
\begin{proof}
 Letting $\partial_t$ act on \eqref{cns.b} and multiplying the resulting identity by $u_{xxt}$ yields
 \begin{align}\label{id5.1}
 \left[\frac{1}{2}u_{xt}^2\right]_t-(u_{xt}u_{tt})_x+\frac{\nu r^{2m}u_{xxt}^2}{\tau}=u_{xxt}R_u
 \end{align}
 with 
 \begin{align*}
 R_u:=~&-\left(\frac{v^2}{r}\right)_t+\left(r^mP_x\right)_t+2m\left(r^{m-1}u\mu_x\right)_t-(r^m)_t\left[\frac{\nu (r^m u)_x}{\tau}\right]_x\\
 &-r^m\left[\left(\frac{\nu}{\tau}\right)_x(r^mu)_x\right]_t-r^m\left(\frac{\nu}{\tau}\right)_t(r^mu)_{xx}-\frac{r^m\nu}{\tau}\left[(r^mu)_{xxt}-r^m u_{xxt}\right].
 \end{align*}
 Thanks to \eqref{bdy}, we integrate \eqref{id5.1} and use Cauchy's inequality to have
 \begin{align}\label{est5.1}
 \|u_{xt}(t)\|^2+\int_{0}^{t}\|u_{xxt}\|^2\lesssim 
 \|u_{xt}|_{t=0}\|^2
 +\int_{0}^{t}\int_{\mathcal{I}}R_u^2.
 \end{align}
 By virtue of the chain rule, \eqref{r_eq}, and \eqref{cns.b}, for a general smooth function $f(\tau,\theta)$, we have
 \begin{align} \notag
 f(\tau,\theta)_{xt} \lesssim ~&|(\tau_{xt},\theta_{xt})|+|(\tau_x,\tau_t,\theta_x,\theta_t)|^2,\\ \notag
 \left|(r^mu)_{xxt}-r^m u_{xxt}\right| \lesssim ~&
 \left|\left((r^m)_{xxt}u,(r^m)_{xt}u_x,(r^m)_{xx}u_t,(r^m)_{x}u_{xt}, (r^m)_{t}u_{xx}\right)\right|\\
 \lesssim ~&\left|\left(u_{xx},u_x,mu,\tau_x\right)\right| +\left|\left(u_x^2,u_x\right)\right|+\left|\left(\tau_x u_t,u_t\right)\right|+\left|\left(u_{xt},u_{xx}\right)\right|, \notag
 \end{align}
 and
 \begin{align} \notag
 |R_u| \lesssim ~& |(v_t,v)|+|(\theta_x,\tau_x)|+|(\tau_{xt},\theta_{xt})|+
 |(\tau_x,\tau_t,\theta_x,\theta_t)|^2+|u_t\theta_x|\\ \notag
 &+ |(u_t,v,\theta_x,\tau_x)|+|(\theta_x,\tau_x)\tau_{tt}|
 +|(\tau_{xt},\theta_{xt})||\tau_t|+
 |(\tau_x,\tau_t,\theta_x,\theta_t)|^2|\tau_t|\\ \notag
 &+|(\theta_t,\tau_t)\tau_{xt}|+
 \left|\left(u_{xx},u_x,mu,\tau_x,u_x^2,\tau_x u_t,u_t,u_{xt}\right)\right|.
 \end{align}
 Applying Sobolev's inequality, we get from
 \eqref{E_1a}, \eqref{E_wx}, \eqref{E_uxvx}, \eqref{E_1b}, \eqref{E_uxx}, and \eqref{E_tauxx} that
 \begin{align} \label{est_taux}
 \lrn (\tau_t,\tau_x,u_x,v_x,w_x)\rrn\lesssim 1.
 \end{align}
Using \eqref{apriori5}, \eqref{est_taux}, Corollary \ref{cor_1}, and Lemmas \ref{lem_wx}--\ref{lem_tauxx}, we derive
\begin{align*}
\int_{0}^{t}\int_{\mathcal{I}}R_u^2\lesssim 1.
\end{align*}
Insert the last estimate into \eqref{est5.1} to get
\begin{align*}
 \|u_{xt}(t)\|^2+\int_{0}^{t}\|u_{xxt}\|^2\lesssim 1.
\end{align*} 
 The estimates for $v$ and $w$ in \eqref{E_uxt} can be obtained in a similar way. 

 We next show the estimate for $u_{xxx}$ in \eqref{E_uxxx}.
 In light of \eqref{est_uxxx}, we have
   \begin{align}
   \notag
   \|u_{xxx}(t)\|^2
   \lesssim 1+\|\theta_{xx}\tau_t\|^2+\|\tau_{xx}\tau_t\|^2
   +\|\tau_{tx}\tau_x\|^2\lesssim  1.
   \end{align}
The other estimates in \eqref{E_uxxx} can be proved similarly by using the system \eqref{cns}.  
The proof of this lemma is completed.
\end{proof}

In the following lemma, we deduce the $L^{\infty}(0,T;L^2(\mathcal{I}))$-norm for $\sqrt{c_v}\theta_{xt}$,
and we can have the bound for $\theta_{xxx}/\sqrt{c_v}$ in $L^{\infty}(0,T;L^2(\mathcal{I}))$ due to the equation \eqref{cns.e}. 
\begin{lemma} \label{lem_thext}
If the conditions listed in Lemma \ref{lem_bas} hold for a sufficiently small $\epsilon_1$, then
 \begin{equation} \label{E_thext}
 \sup_{t\in[0,T]}\left\|\left(\sqrt{c_v}\theta_{xt},\frac{\theta_{xxx}}{\sqrt{c_v}}\right)(t)
 \right\|^2 +\int_{0}^{T}
 \|\theta_{xxt}(t)\|^2\mathrm{d}t\leq C_{10}.
 \end{equation}
\end{lemma}
\begin{proof}
 Let $\partial_t$ act on \eqref{cns.e} and multiply the resulting identity by $\theta_{xxt}$ to get
 \begin{align}
 \left[\frac{c_v}{2}\theta_{xt}^2\right]_t
 -\left(c_v\theta_{tt}\theta_{xt}\right)_x+\frac{\kappa r^{2m}\theta_{xxt}^2}{\tau}= \theta_{xxt}R_{\theta}, \label{id5.2}
 \end{align}
 with 
 \begin{align*}
 R_{\theta}:=\left[P(r^m u)_x\right]_t
 -\left[\left(\frac{\kappa r^{2m}\theta_x}{\tau}\right)_{xt}
 -\frac{\kappa r^{2m}\theta_{xxt}}{\tau}\right]-\mathcal{Q}_t.
 \end{align*}
 Integrating the last identity and using the boundary condition for $\theta$ yield
 \begin{align}
 \|\sqrt{c_v}\theta_{xt}(t)\|^2+\int_{0}^{t}\|\theta_{xxt}(s)\|^2\mathrm{d}s
 \lesssim  \|\sqrt{c_v}\theta_{xt}|_{t=0}\|^2
 +\int_{0}^{t}\int_{\mathcal{I}}R_{\theta}^2. \label{est_thext}
 \end{align}
By virtue of \eqref{est_thext1} and \eqref{est_Qx2}, we have
\begin{align*}
\|\sqrt{c_v}\theta_{xt}|_{t=0}\|^2\lesssim 1.
\end{align*}
To estimate the last term on the right, we compute
\begin{align} \label{est5.2}
|(P(r^m u)_x)_t|\lesssim~& |\tau_t(\theta_t,\tau_t)|+|\tau_{tt}|,\\
|\mathcal{Q}_t|\lesssim~& |(\theta_t,\tau_t)\tau_t^2|+|\tau_t\tau_{tt}|+
|\theta_t(r^{m-1}u^2)_x|+|(r^{m-1}u^2)_{xt}|+|w_xw_{xt}|\notag\\
&+ \left|\left(\frac{\mu r^{2m}}{\tau}\right)_tw_x^2\right|
+|(\theta_t,\tau_t)||(v_x,v)|^2+|(v_x,v)||(v_{xt},v_x,v_t,v)|,\label{est5.3}
\end{align}
and
 \begin{align}\notag
 &\left|\left[\left(\frac{\kappa r^{2m}\theta_x}{\tau}\right)_{xt}
 -\frac{\kappa r^{2m}\theta_{xxt}}{\tau}\right]\right|\\[1mm]
&\lesssim  |(\theta_x,\tau_x,1)\theta_{xt}|+ |(\theta_t,\tau_t,1)\theta_{xx}|+
 |\theta_x||(mu,u_x)|\notag \\[1mm]
 &\quad +|\theta_x|\left[|(\theta_x,\tau_x)|+|(\theta_t,\tau_t)|
 +|(\theta_{xt},\tau_{xt})|+
 |(\theta_x,\theta_t,\tau_x,\tau_t)|^2\right]. \label{est5.4}
 \end{align}
 Plug \eqref{est5.2}--\eqref{est5.4} into \eqref{est_thext} and use 
 \eqref{apriori5}, \eqref{E_uvw}, \eqref{est_taux}, Corollary \ref{cor_1},  and Lemmas \ref{lem_wx}--\ref{lem_uxt} to conclude 
  \begin{equation} \label{E_thext1}
  \sup_{t\in[0,T]}\left\|\sqrt{c_v}\theta_{xt}(t)
  \right\|^2 +\int_{0}^{T}
  \|\theta_{xxt}(t)\|^2\mathrm{d}t\lesssim 1.
  \end{equation}
The estimate \eqref{E_thext} follows by using \eqref{est_thexxx}, \eqref{est_Qx2} and \eqref{E_thext1}.
\end{proof}

By using the system \eqref{cns}, we can get the following estimates for $(u_{xxxx},v_{xxxx},w_{xxxx},\theta_{xxxx})$.
The proof is similar to that of Lemma \ref{lem_iuxxx} and hence we omit the details for brevity.
\begin{lemma} If the conditions listed in Lemma \ref{lem_bas} hold for a sufficiently small $\epsilon_1$, then
 \begin{equation} \label{E_iuxxxx}
 \int_{0}^{T}
 \left\|\left(u_{xxxx},v_{xxxx},w_{xxxx},
 \frac{\theta_{xxxx}}{c_v}\right)(t)\right\|^2
 \mathrm{d}t\lesssim 1
 +\int_{0}^{T}\|\tau_{xxx}(t)\|^2\mathrm{d}t.
 \end{equation}
\end{lemma}

The following lemma is to establish the bound for $L_x^2$-norm of $\tau_{xx}(t,x)$. We note that, unlike the above estimates, the bound in \eqref{E_tauxxx} depends on $\gamma-1$. 
\begin{lemma} \label{lem_tauxxx}
 If the conditions listed in Lemma \ref{lem_bas} hold for a sufficiently small $\epsilon_1$, then
  \begin{align} \label{E_tauxxx}
  \sup_{t\in[0,T]}\left\|\tau_{xxx}(t)\right\|^2 +
  \int_{0}^{T}
  \left\|\left(\tau_{xxx},u_{xxxx},v_{xxxx},w_{xxxx},
  \frac{\theta_{xxxx}}{\sqrt{c_v}}\right)(t)\right\|^2
  \mathrm{d}t\lesssim c_v.
  \end{align}
\end{lemma}
\begin{proof}
Let $\partial_x^2$ act on \eqref{id2.2} and multiply the resulting identity by $(\nu \tau_x/\tau)_{xx}$ to get
\begin{align}\label{id_tauxxx}
\left[\frac{1}{2}\left(\frac{\nu \tau_x}{\tau}\right)_{xx}^2\right]_t
+\frac{\theta}{\nu \tau}\left[\frac{\nu \tau_x}{\tau}\right]_{xx}^2
=\left[\frac{\nu \tau_x}{\tau}\right]_{xx}R_{\tau},
\end{align}
with
\begin{align*}
R_{\tau}=\left[\frac{u_t}{r^{m}}-\frac{v^2}{r^{m+1}}
+\frac{2m u\mu_x}{r}+\frac{\nu'(\theta)}{\tau}\left(\tau_x\theta_t-\tau_t\theta_x\right)+\left(\frac{\theta}{\tau}\right)_{x}\right]_{xx}+\frac{\theta}{\nu\tau}\left(\frac{\nu \tau_x}{\tau}\right)_{xx}.
\end{align*}
Integrating \eqref{id_tauxxx} and using Cauchy's inequality yield
\begin{align}
\left\|\left(\frac{\nu \tau_x}{\tau}\right)_{xx}(t)\right\|^2
+\int_{0}^{t}\left\|\left(\frac{\nu \tau_x}{\tau}\right)_{xx}(s)\right\|^2\mathrm{d}s
\lesssim 1+\int_{0}^{t}\int_{\mathcal{I}}R_{\tau}^2. \label{est5.5}
\end{align}	
From \eqref{E_uxxx} and \eqref{E_uxx}, we have
\begin{align}
\lrn (\tau_{xt},u_{xx},v_{xx},w_{xx})\rrn\lesssim 1. \label{est_tauxt}
\end{align}
Using \eqref{est_taux} and \eqref{est_tauxt}, 
after a direct calculation, we have
\begin{align}\notag
R_{\tau}^2\lesssim ~&
|(u_t,u_{xt},u_{xxt})|^2+|(v,v_x,v_{xx})|^2+|(\theta_x,\theta_{xx},\theta_{xxx})|^2\\
&+|(\tau_x,\theta_x,\tau_{xx},\theta_{xx},\theta_{xxx})|^2
+|(\tau_{xxx}\theta_t,\tau_{xxt},\tau_{xx}\theta_{xt},\theta_{xxt})|^2. \label{est_Rtau}
\end{align}
Plug \eqref{est_Rtau} into \eqref{est5.5} and use 
Corollary \ref{cor_1}, Lemmas \ref{lem_wx}--\ref{lem_thext} to get
\begin{align}\notag
&\left\|\left(\frac{\nu \tau_x}{\tau}\right)_{xx}(t)\right\|^2
+\int_{0}^{t}\left\|\left(\frac{\nu \tau_x}{\tau}\right)_{xx}(s)\right\|^2\mathrm{d}s\\ \notag
&\quad\lesssim 
\left\|\left.\left(\frac{\nu \tau_x}{\tau}\right)_{xx}\right|_{t=0}\right\|^2
+\int_{0}^{t}\|\theta_{xxx}\|^2+
\lrn\theta_t\rrn^2\int_{0}^{t}\|\tau_{xxx}\|^2+
\int_{0}^{t}\int_{\mathcal{I}}\tau_{xx}^2\theta_{xt}^2.  
\end{align}
By virtue of \eqref{apriori1}, \eqref{apriori2} and \eqref{E_tauxx}, we have
\begin{align}
\left\|\left(\frac{\nu \tau_x}{\tau}\right)_{xx}(t)\right\|^2
+\int_{0}^{t}\left\|\left(\frac{\nu \tau_x}{\tau}\right)_{xx}(s)\right\|^2\mathrm{d}s
\lesssim c_v+
\epsilon_1\left[\int_{0}^{t}\|\tau_{xxx}\|^2+\sup_{s\in[0,t]}\|\tau_{xx}(s)\|^2\right]. \label{est5.6}
\end{align}
Employ \eqref{est_taux} and \eqref{est_tauxt} to have
\begin{align*}
\left|\frac{\nu\tau_{xxx}}{\tau}\right|
\lesssim \left|\left(\frac{\nu \tau_x}{\tau}\right)_{xx}\right|+
\left|(\tau_{xx},\theta_{xx},\tau_x,\theta_x)\right|,
\end{align*}
which combined with \eqref{est5.6} and \eqref{E_iuxxxx} implies \eqref{E_tauxxx} if $\epsilon_1$ is sufficiently small.
\end{proof}

\section{Proof of Theorem \ref{thm}}
\label{sec3}
In this section we complete the proof of Theorem \ref{thm} by combining the {\em a priori} bounds obtained in Section \ref{sec2}, the continuation argument and Poincar\'{e}'s inequality. 
We divide this section into two parts. 
The first is devoted to proving the existence and uniqueness of global solutions to the problem \eqref{cns}--\eqref{bdy}, and the second one  aims at showing the convergence decay rate of the solutions toward constant states.

\subsection{Global solvability}
We first present the local solvability result to the initial boundary value problem \eqref{cns}--\eqref{bdy} in the following proposition, which can be proved by the standard iteration method (see, for instance, \citet{AKM90MR1035212}).

\begin{proposition} \label{local}
Let the initial data $(\tau_0,u_0,v_0,w_0,\theta_0)\in H^3(\mathcal{I})$ satisfy that
\begin{align*}
&\tau_0(x)\geq \lambda_1^{-1},\quad 
\theta_0(x)\geq \lambda_2^{-1}\quad
\forall\ x\in\mathcal{I},\\[0.5mm]
&\|u_0\|_{1}+\sqrt{c_v}\|(\theta_0-1,\theta_{t}|_{t=0})\|_1+\|\theta_{0xx}\|\leq \Lambda,\\
&\|(\tau_0,u_0,v_0,w_0)\|_{3}+\sqrt{c_v}\|\theta_0-1\|_1+\|(\theta_{0xx},\sqrt{\gamma-1}\theta_{0xxx})\|\leq \Pi,
\end{align*}
for some positive constants  $\lambda_1$, $\lambda_2$, $\Lambda$, and $\Pi$, where $\theta_t|_{t=0}$ and $\theta_{xt}|_{t=0}$ are defined by \eqref{initial_thet} and \eqref{initial_thext}, respectively.
Then there exists a positive constant $T_0=T_0(\lambda_1,\lambda_2,\Lambda,\Pi)$, which depends only on $\lambda_1$, $\lambda_2$, $\Lambda$, and $\Pi$, such that the initial boundary value problem \eqref{cns}--\eqref{bdy} has a unique solution $(\tau, u, v, w, \theta) \in X(0,T_0;2\lambda_1,2\Lambda)$.
\end{proposition}

According to \eqref{H1} and \eqref{H2}, there exists a $(\gamma-1)$-independent positive constant $C_0$, such that
 \begin{align*}
 & \tau_0(x)\geq V_0^{-1},\quad 
 \theta_0(x)\geq V_0^{-1}\quad
 \forall\ x\in\mathcal{I},\\[1mm]
 &\|u_0\|_{1}+\sqrt{c_v}\|(\theta_0-1,\theta_{t}|_{t=0})\|_1+\|\theta_{0xx}\|\leq C_0,\\
 &\|(\tau_0,u_0,v_0,w_0)\|_{3}+\sqrt{c_v}\|\theta_0-1\|_1+\|(\theta_{0xx},\sqrt{\gamma-1}\theta_{0xxx})\|\leq \Pi_0.
 \end{align*}
Let $\epsilon_1$, $C_j$ ($j=1,\cdots,10$) be chosen in Section \ref{sec2}.
We assume that $\gamma-1\leq \epsilon_0$ with
\begin{align}\label{epsil0}
&\epsilon_0:=\min\{\delta_1,\delta_2\},
\end{align}
where
\begin{align} \label{def3.1}
\delta_1:=\left[\frac{\epsilon_1}{(2C_0)^2(2V_0)^2}\right]^4,\quad
\delta_2:=\left[\frac{\epsilon_1}{5C_{11}(2C_1)^2 }\right]^4,\quad
C_{11}:=\sum_{j=1}^{10}C_j.
\end{align}

Applying  Proposition \ref{local}, we can find a positive constant $t_1=T_0(V_0,V_0,C_0,\Pi_0)$ such that there exists a unique solution  $(\tau, u, v, w, \theta) \in X(0,t_1;2V_0,2C_0)$ to the initial boundary value problem \eqref{cns}--\eqref{bdy}.

Since $\gamma-1\leq \delta_1$,
we can apply  Corollary \ref{cor_1} and Lemmas \ref{lem_wx}--\ref{lem_tauxxx} with $T=t_1$ to deduce that there exists a positive constant $C_1(\gamma)$, depending only on $V_0$, $\Pi_0$, and $\gamma$, such that for each $t\in[0,t_1]$,  the local solution  $(\tau, u, v, w, \theta) $ satisfies
 \begin{align} \label{pr1}
 &\mathcal{E}_0(t)\leq C_{11},\quad 
  \tau(t,x)\geq C_1^{-1},\quad 
 \theta(t,x)\geq \tfrac{1}{2}\quad
 \forall\ x\in\mathcal{I},\\ \label{pr2}
 &\|(\tau,u,v,w)(t)\|_{3}+\sqrt{c_v}\|(\theta_0-1)(t)\|_1+\|(\theta_{xx},\sqrt{\gamma-1}\theta_{xxx})(t)\|\leq C_1(\gamma).
 \end{align}

If we take  $(\tau, u, v, w, \theta) (t_1,\cdot)$ as the initial data and apply Proposition \ref{local} again,
we can extend the local solution  $(\tau, u, v, w, \theta)$ to the time interval $[0,t_1+t_2]$ with
$t_2=T_0(C_1,2,C_{11},C_1(\gamma))$ such that $(\tau, u, v, w, \theta)\in X(t_1,t_1+t_2;2C_1,2\sqrt{C_{11}})$.
Hence the local solution $(\tau, u, v, w, \theta)\in X(0,t_1+t_2;2C_1,\sqrt{5C_{11}})$. 
Noting that $\gamma-1\leq \delta_2$, we can apply  Corollary \ref{cor_1} and Lemmas \ref{lem_wx}--\ref{lem_tauxxx} with $T=t_1+t_2$ to deduce that \eqref{pr1} and \eqref{pr2} hold for each $t\in[0,t_1+t_2]$.

Repeating the above procedure, we can then extend the solution $(\tau, u, v, w, \theta)$ step by step to a global one provided that  $\gamma-1\leq \epsilon_0$ with $\epsilon_0$ given by \eqref{epsil0}. Furthermore,
\begin{align*}
\|(\tau,u,v,w,\theta)(t)\|_{3}^2+\int_0^{\infty}\left[\|(u_x,v_x,w_x,\theta_x)\|_3^2+\|\tau_{x}\|_2^2\right]\mathrm{d}t\leq C_2(\gamma)
\qquad \forall\ t\in[0,\infty),
\end{align*}
where $C_2(\gamma)$ is some positive constant depending on $\gamma$, $\Pi_0$, and $V_0$.

\subsection{Convergence decay rate}
It follows from \eqref{id_intau} and Poincar\'{e}'s inequality that
\begin{align} \label{pr2.1}
	\|(\tau-\bar{\tau})(t)\|^2=
	\int_{\mathcal{I}}\left|\tau(t,x)-\int_{\mathcal{I}}\tau(t,y)\mathrm{d}y\right|^2\mathrm{d}x\lesssim \|\tau_{x}(t)\|^2.
\end{align}
Thanks to the boundary conditions for $(u,v,w)$, we apply Poincar\'{e}'s inequality to get
\begin{align}\label{pr2.2}
\|(u,v,w)(t)\|\lesssim \|(u_x,v_x,w_x)(t)\|.
\end{align}
Multiplying \eqref{cns.b}, \eqref{cns.c}, \eqref{cns.d} by $u$, $v$, $w$, respectively,  adding them with \eqref{cns.e}, and integrating the resulting identity over $[0,t]\times\mathcal{I}$ yield
\begin{align*}
\int_{\mathcal{I}}\left[c_v\theta+\frac{1}{2}(u^2+v^2+w^2)\right](t,x)\mathrm{d}x=c_v\bar{\theta}\qquad \forall\ t\in[0,\infty),
\end{align*}
where $\bar{\theta}$ is given by \eqref{hat}.
We use Poincar\'{e}'s inequality and  \eqref{pr2.2} to obtain
\begin{align}\label{pr2.4}
\|(\theta-\bar{\theta})(t)\|^2\leq 
\int_{\mathcal{I}}\left|\theta(t,x)-\int_{\mathcal{I}}\theta(t,y)\mathrm{d}y\right|^2\mathrm{d}x+\|(u,v,w)(t)\|^2
\lesssim \|(u_x,v_x,w_x,\theta_x)(t)\|^2.
\end{align}
We have from  \eqref{r_eq} that for all $(t,x)\in[0,\infty)\times\mathcal{I}$,
\begin{align*}
\left[r^{m+1}(t,x)\right]_x=(m+1)\tau(t,x),\quad
r^{m+1}(t,x)=a^{m+1}+(m+1)\int_0^x\tau(t,y)\mathrm{d}y.
\end{align*}
Then 
\begin{align*}
\left|r^{m+1}(t,x)-\bar{r}^{m+1}\right|&=
(m+1)\left|\int_0^x(\tau(t,y)-\bar{\tau})\mathrm{d}y\right|,\\
\left[r^{m+1}(t,x)-\bar{r}^{m+1}\right]_x&=(m+1)\left[\tau(t,x)-\bar{\tau}\right],
\end{align*}
which combined with \eqref{pr2.1} imply
\begin{align} \label{pr2.15}
\|r(t)-\bar{r}\|_2\lesssim \left\|r^{m+1}(t)-\bar{r}^{m+1}\right\|_2
\lesssim \left\|\tau(t)-\bar{\tau}\right\|_1.
\end{align}

To derive the exponential stability of the solution $(\tau,u,v,w,\theta)$, we define the following energy functionals:
\begin{align*}
\mathcal{H}_{\eta}(t)&:=\int_{\mathcal{I}}\eta_{\bar{\theta}}(\tau,u,v,w,\theta)(t,x)\mathrm{d}x,\\
\quad\mathcal{H}_{\tau}(t)&:=\int_{\mathcal{I}}\tau_x^2(t,x)\mathrm{d}x,\quad \mathcal{H}_{\theta}(t):=\int_{\mathcal{I}}\theta_x^2(t,x)\mathrm{d}x,\\
\mathcal{H}_{U}(t)&:=\int_{\mathcal{I}}\left[u_x^2+v_x^2+w_x^2\right](t,x)\mathrm{d}x,\\
\quad\mathcal{H}_{0}(t)&:=\int_{\mathcal{I}}\left[\frac{1}{2}\left(\frac{\nu \tau_x}{\tau}\right)^2-
\frac{u}{r^{m}}\frac{\nu \tau_x}{\tau}\right](t,x)\mathrm{d}x,
\end{align*}
where $\eta_{\bar{\theta}}(\tau,u,v,w,\theta)$ is defined by \eqref{eta}. 
Then we have from \eqref{pr2.2} that
\begin{align}
\label{pr2.13}
\mathcal{H}_{\tau}(t)-C\|u(t)\|^2\lesssim \mathcal{H}_{0}(t)\lesssim \mathcal{H}_{\tau}(t)+\|u(t)\|^2
\lesssim  \mathcal{H}_{\tau}(t)+ \mathcal{H}_{U}(t).
\end{align}

It follows from \eqref{est3}, \eqref{est4}, and \eqref{pr2.2} that
\begin{align}\label{pr2.5}
\frac{\mathrm{d}}{\mathrm{d}t}\mathcal{H}_{U}(t)+c\|(u_{xx},v_{xx},w_{xx})(t)\|^2\lesssim 
\mathcal{H}_{\tau}(t)+\mathcal{H}_{U}(t)+\mathcal{H}_{\theta}(t).
\end{align}
In light of \eqref{est5} and \eqref{pr2.2}, we have
\begin{align}\label{pr2.6}
c_v\frac{\mathrm{d}}{\mathrm{d}t}\mathcal{H}_{\theta}(t)+c\|\theta_{xx}(t)\|^2\lesssim 
\mathcal{H}_{\tau}(t)+\mathcal{H}_{U}(t)+\mathcal{H}_{\theta}(t)+\|(u_{xx},v_{xx},w_{xx})(t)\|^2.
\end{align}
Multiplying \eqref{pr2.6} by a sufficiently small $\sigma_1>0$ and adding this with \eqref{pr2.5} yield
\begin{align}
\frac{\mathrm{d}}{\mathrm{d}t}
\left[\mathcal{H}_{U}(t)+c_v\sigma_1\mathcal{H}_{\theta}(t)\right]+c\sigma_1\|(u_{xx},v_{xx},w_{xx},\theta_{xx})(t)\|^2
\lesssim 
\mathcal{H}_{\tau}(t)+\mathcal{H}_{U}(t)+\mathcal{H}_{\theta}(t). \label{pr2.7}
\end{align}
By virtue of \eqref{est2}, \eqref{pr2.2}, and \eqref{est2a},  we get
\begin{align} \notag
\frac{\mathrm{d}}{\mathrm{d}t}\mathcal{H}_{0}(t)+c\mathcal{H}_{\tau}(t)
\lesssim~&\left[1+(\gamma-1)^2\right]\left[\mathcal{H}_{U}(t)+\mathcal{H}_{\theta}(t)\right]\\
&+(\gamma-1)^2\left[
\mathcal{H}_{\tau}(t)+\|(u_{xx},v_{xx},w_{xx},\theta_{xx})(t)\|^2\right]. \label{pr2.8}
\end{align}
Since $c_v^{-1}=|\gamma-1|\leq \epsilon_0$,
we multiply \eqref{pr2.7} by a sufficiently small $\sigma_2$ and add this into \eqref{pr2.8} to find
\begin{align} \notag
&\frac{\mathrm{d}}{\mathrm{d}t}
\left[\mathcal{H}_{0}(t)+\sigma_2\mathcal{H}_{U}(t)+c_v\sigma_2\sigma_1\mathcal{H}_{\theta}(t)\right]+c\mathcal{H}_{\tau}(t)+c\sigma_2\sigma_1\|(u_{xx},v_{xx},w_{xx},\theta_{xx})(t)\|^2
\\ &\quad \lesssim \left[1+\epsilon_0^2\right]\left[\mathcal{H}_{U}(t)+\mathcal{H}_{\theta}(t)\right]+\epsilon_0^2\left[
\mathcal{H}_{\tau}(t)+\|(u_{xx},v_{xx},w_{xx},\theta_{xx})(t)\|^2\right]. \label{pr2.9}
\end{align}
Take $\epsilon_0>0$ suitably small to infer
\begin{align} 
\frac{\mathrm{d}}{\mathrm{d}t}
\left[\mathcal{H}_{0}(t)+\sigma_2\mathcal{H}_{U}(t)+c_v\sigma_2\sigma_1\mathcal{H}_{\theta}(t)\right]+c\mathcal{H}_{\tau}(t)
\lesssim \mathcal{H}_{U}(t)+\mathcal{H}_{\theta}(t). \label{pr2.10}
\end{align}
It follows from \eqref{est1} and that
\begin{align} \label{pr2.11}
 \frac{\mathrm{d}}{\mathrm{d}t}\mathcal{H}_{\eta}(t)+c\left[\mathcal{H}_{U}(t)+\mathcal{H}_{\theta}(t)\right]\leq 0.
\end{align}
Multiplying \eqref{pr2.10} by a sufficiently small $\sigma_3$ and adding this into \eqref{pr2.11} imply
\begin{align}
\frac{\mathrm{d}}{\mathrm{d}t}
\mathcal{H}(t)+c\left[\mathcal{H}_{\tau}(t)+\mathcal{H}_{U}(t)+\mathcal{H}_{\theta}(t)\right]
\leq 0, \label{pr2.12}
\end{align}
where 
\begin{align*}
\mathcal{H}(t):=\mathcal{H}_{\eta}(t)+\sigma_3\mathcal{H}_{0}(t)+\sigma_3\sigma_2\mathcal{H}_{U}(t)+c_v\sigma_3\sigma_2\sigma_1\mathcal{H}_{\theta}(t).
\end{align*}
Using \eqref{pr2.13}, \eqref{pr2.1}, \eqref{pr2.2}, and \eqref{pr2.4}, we have
\begin{align*}
\left[\frac{1}{2}-\sigma_3C\right]\|u(t)\|^2+c\sigma_3\left[\mathcal{H}_{\tau}(t)+\mathcal{H}_{U}(t)+\mathcal{H}_{\theta}(t)\right]\leq \mathcal{H}(t)\lesssim \mathcal{H}_{\tau}(t)+\mathcal{H}_{U}(t)+\mathcal{H}_{\theta}(t).
\end{align*}
Take $\sigma_3>0$ with $C\sigma_3\leq 1/4$ so that
\begin{align*}
c\sigma_3\left[\mathcal{H}_{\tau}(t)+\mathcal{H}_{U}(t)+\mathcal{H}_{\theta}(t)\right]\leq \mathcal{H}(t)\lesssim \mathcal{H}_{\tau}(t)+\mathcal{H}_{U}(t)+\mathcal{H}_{\theta}(t),
\end{align*} 
which combined with \eqref{pr2.12} yields
\begin{align} \label{pr2.14}
\mathcal{H}_{\tau}(t)+\mathcal{H}_{U}(t)+\mathcal{H}_{\theta}(t)\lesssim \mathcal{H}(t)\lesssim \mathrm{e}^{-ct}.
\end{align}
We combine \eqref{pr2.14}, \eqref{pr2.1}, \eqref{pr2.2}, \eqref{pr2.4}, and \eqref{pr2.15} to conclude the exponential decay estimate \eqref{thm_C3}. Therefore, the proof of Theorem \ref{thm} has been completed.

\bigbreak

\begin{center}
	{\bf Acknowledgement}
\end{center}
Research of the authors was supported by the Fundamental Research Funds for the Central Universities, the Project funded by China Postdoctoral Science Foundation, and the grants from National Natural Science Foundation of China under contracts 11601398 and 11671309.
The authors express much gratitude to Huijiang Zhao for his support and advice.
Tao Wang would like to warmly thank Paolo Secchi, Alessandro Morando, and Paola Trebeschi for support and hospitality during his postdoctoral stay at University of Brescia.

\bibliographystyle{abbrvnat}

\bibliography{YJDEQ8706}

\end{document}